\documentclass[11pt,a4paper]{article}

\usepackage{amsmath}
\usepackage{amsthm}
\usepackage{amsfonts}
\usepackage{amssymb}
\usepackage{mathtools}
\usepackage{dsfont}

\usepackage[
    hmargin=2.5cm,
    vmargin=3cm
]{geometry}

\usepackage{hyperref}

\theoremstyle{plain}
\newtheorem{theorem}{Theorem}[section]
\newtheorem{lemma}[theorem]{Lemma}

\theoremstyle{definition}
\newtheorem{definition}[theorem]{Definition}
\newtheorem{example}[theorem]{Example}

\theoremstyle{remark}
\newtheorem{remark}[theorem]{Remark}

\title{Functional limit theorems for perturbed random walks}

\author{
Andrey Pilipenko\thanks{
Section de Mathématiques,
Université  de Genève,
CP 64,
1211 Genève 4
Switzerland;
and
Institute of Mathematics, National Academy of Sciences of Ukraine,
Tereshchenkivska Str.\ 3, 01601 Kyiv, Ukraine.
}
}

\date{}

\begin{document}

\maketitle

\begin{abstract}
We consider scaling limits of integer-valued random walks whose transition
probabilities are translation-invariant everywhere except on a set that we
will call a membrane or an interface. If the jumps outside the membrane are
i.i.d.\ mean-zero random variables with finite variance, then possible
examples of the corresponding limit processes include reflected Brownian
motion, skew Brownian motion, and Brownian motion with a jump-type exit from
zero. We also discuss the question of what the possible limits are if the
jumps outside the membrane belong to the domain of attraction of a stable
law.
\end{abstract}

\section{Introduction}

Consider a random walk $S_\xi(n):=\xi_1+\dots+\xi_n$, where $(\xi_n)$ are independent identically distributed random variables. If  $(\xi_n)$ have zero mean and   finite variance  $\sigma^2>0$, then certain scaling limits  of the random walk converge to Brownian motion.  In this article, we consider a Markov chain on ${\mathbb Z}$ that acts as a perturbation of an underlying random walk $S_\xi$ on some set $A\subset {\mathbb Z}$, which we refer to as a \textit{membrane} or an $\textit{interface}$. Specifically, we assume that the transition probabilities of the perturbed walk and the original walk may differ only within $A$. For instance, inside the membrane, the walk might be more biased toward moving to the right than to the left. This model serves as a macroscopic representation of a {``breathing material,''} which allows diffusion to pass more easily in one direction than the other. The central problem addressed in this review is the identification of the limits for such perturbed random walks using the same normalization as   for standard random walks.

The stated problem can be generalized in several ways. For example, if the jumps to the left and right of the membrane follow different distributions, we arrive at the problem of investigating  {perturbed oscillating random walks}. Another generalization involves  {random walks on star graphs}, where the oscillating random walk can be viewed as a special case.

It is worth noting that if the membrane is bounded, its scalings contract to the origin. Consequently, it is natural to expect that the limiting process (if it exists) will be a Markov process that behaves like Brownian motion away from zero.  For a comprehensive analytical and stochastic description of all such processes, see \cite{Bobrowski-Pilipenko_2025_Walsh}. In particular, if the jumps within the membrane have finite expectation, ensuring the limiting process is continuous, one naturally expects it to be a  {skew Brownian motion} (on the real line) or a  {Walsh process} (on a graph). The case of heavy tails outside and/or inside the membrane is significantly more challenging.

The purpose of this article is to discuss methods for proving the existence of the corresponding limits, as well as ways to describe these limits. We will mainly focus on discussing ideas, intuition, the main steps of proofs, difficulties that may be encountered, and methods for overcoming them. For detailed proofs, see the recent monograph \cite{IPMS_book}. %While preparing this article, we identified several alternative investigative approaches, critical nuances, and potential generalizations to which we wish to draw the reader's attention.

%\subsection*{Organization of the Article}

The article is organized as follows. In Section \ref{sec:2formulation}, we provide a mathematical formulation of the problem and several illustrative examples leading to reflected Brownian motion and skew Brownian motion. We also provide various descriptions of these processes, as they shed light on potential research methodologies.

Section \ref{sec:firststeps} is devoted to the first steps of analyzing the limit behavior of perturbed random walks. We prove several general theorems under minimal assumptions on the transition probabilities of the perturbed walk. For example, we prove that the number of visits to the membrane during $n$ steps is $o(n)$. This allows us to make an informal guess for the stochastic representation of the limit process, expressed as a stochastic differential equation with local time, see \eqref{eq:guessLimit}.
  This provides further support for our earlier conjecture that the limit of a perturbed random walk on the line is a skew Brownian motion, provided the jumps outside and inside the membrane have finite expectation. Although this answer was easy to guess from the outset (and was originally suggested as a ``simple problem'' to my PhD student, Yuri Prykhodko), obtaining a formula for the parameters of skew Brownian motion in a general setting took nearly a decade.

If the jumps outside and/or inside the membrane belong to the domain of attraction of a stable law, then the processes $U$ and/or $V$ in equation \eqref{eq:guessLimit} involving local time are stable processes, which is a highly non-standard situation.
We leave it to the reader until \S \ref{sec:skew Levy}  to think about how such processes can be described, not to mention the difficulties associated with proving convergence.

In Section \ref{sec:reflection}, we study random walks with a reflective membrane and prove the invariance principle. In some cases, the result leads to Skorokhod's reflection problem and is a simple consequence of the continuous mapping theorem. However, in cases where jumps out of the membrane are heavy-tailed, the result of the scaled limit will differ, leading to a Brownian motion on a half-line with jump-type reflection at $0$, which corresponds to non-local boundary conditions for the corresponding semigroup. The constructions of Section  \ref{sec:reflection} are used later in  Section \ref{sec:spider}, where we consider perturbed random walks on a line or on a star-graph, whose jumps have light tails. The main result is that their scaling limit is a skew Brownian motion or a Walsh process, respectively. The proof of the corresponding result provided here is new; it is based on a recent description of the Walsh process \cite{Bayraktar} and potentially seems to allow for many generalizations.

Section \ref{sec:skew Levy} is devoted to the case where both jumps outside the membrane and jumps from the membrane are heavy-tailed. In this case, the limiting process may be a  {skew stable L\'evy process}, introduced   recently in \cite{IksanovPilipenko2021skewLevy}.

At the end of  sections, we discuss possible generalizations and other approaches to investigation.

\section{Formulation of the problem, warm up examples, and informal reflections}\label{sec:2formulation}

  %  For $t\in(n,n+1)$ we will  extend $S_\xi$ by linearity  $S_\xi(t):=S_\xi(n)+(t-n)(S_\xi(n+1)-S_\xi(n)) $   or we will set   $S_\xi(t):=S_\xi(n).$ In  cases there will be no difference in results

  We will investigate weak convergence of random processes  in the Skorokhod space of 
  %continuous functions $C([0,\infty))$ equipped with the  topology of  uniform convergence on compact sets, or in the space of 
  c\`adl\`ag functions ${\mathcal D}\coloneqq D([0,\infty))$ equipped with  $J_1$-topology\footnote{Actually, A.Skorokhod introduced four topologies, $J_1$ is a `usual' one. }, see for example \cite{Billingsley, Ethier+Kurtz:1986, Whitt} for the definition. Convergence in distribution will be denoted by $\Rightarrow.$ All   sequences $(X(n))_{n\geq 0}$ we will be extended to non-negative real numbers by $X(t)\coloneqq X([t]), t\geq 0,$ or by linearity.

  Consider a random walk (RW) $S_\xi(n)\coloneqq S_\xi(0)+\sum_{k=1}^n\xi_k, n\geq 1,$ where  $(\xi_n)\overset{d}{=} \xi$ are independent identically distributed random variables that are independent of $S_\xi(0)$. 
  Since we expect   existence of a nontrivial scaling limit  for perturbation of   $S_\xi$, it is natural to expect that scaling limits of $S_\xi$ exists itself, i.e.,   that there is a sequence $(a_n)$ such that we have convergence to some process $U,$
  \begin{equation}
      \label{eq:conv_Skor1}
      \Big(\frac{S_\xi(nt)}{a_n}\Big)_{t\geq 0}\Rightarrow \Big(U(t)\Big)_{t\geq 0}, \quad n\to\infty.
    \end{equation}
    For example, if ${\mathop{\rm E}} \xi=0$ and ${\mathop{\rm Var}} \xi=\sigma^2\in(0,\infty),$ then by Donsker's theorem we have convergence of processes
  \[
  \Big(\frac{S_\xi(nt)}{\sigma \sqrt{n}}\Big)_{t\geq 0}\Rightarrow \Big(B(t)\Big)_{t\geq 0}, \quad n\to\infty,
  \]
where $B$ is a standard Brownian motion.
  
Let $A\subset \mathbb{Z}$ be a subset of integers, $\xi$ and $\eta_x, x\in A,$ be integer-valued random variables. 
A Markov chain $(X(n))$ with transition probabilities
\begin{equation}\label{eq:transitions_PRW}
    p_{x,y}=\begin{cases}
    {\mathop{\rm P}}(\xi=y-x), & x\notin A, y\in{\mathbb Z};\\
    {\mathop{\rm P}}(\eta_x=y-x), & x\in A, y\in{\mathbb Z},
\end{cases}
\end{equation}
is called a perturbation of a random walk $S_\xi $ at $A$ or a RW with a membrane $A.$ Formula \eqref{eq:transitions_PRW} means that jumps outside of the set $A$ have distribution $\xi$ and jumps from a point $  x\in A,$ follow the distribution $\eta_x.$  To avoid certain trivial cases we will always assume that   

\textit{all states in ${\mathbb Z}\setminus A$   communicate and the random walk cannot stay in the membrane forever  with probability 1.}

Mostly, we will consider the case when

\textit{  $A$ is a finite set, say $A=\{-m,\dots,m\}$, or the case $A={\mathbb Z}\setminus {\mathbb N}.$ In the first case we  say that $A$ is a semipermeable membrane. In the second case we will assume that $x+\eta_x>0$ and call $A$ a reflective membrane.}

A stochastic representation of $X$ can be constructed as follows. Let
$(\xi_n), (\eta_{n,x})_{n\geq 1}, x\in A,$ be independent sequences of i.i.d.  integer-valued random  variables that are independent of $X(0)$ and  such that $\xi_n\overset{d}{=}\xi$ and $\eta_{n,x}\overset{d}{=}\eta_x,x\in A.$
  Define a sequence of random variables recurrently
  \[
  X(n+1)=X(n)+ \xi_{n+1}{\mathds{1}}_{X(n)\notin A} + \sum_{x\in A} \eta_{n+1,x}{\mathds{1}}_{ X(n) =x}, \ n\geq 0,
  \]
  or
\begin{equation}
    \label{eq:X_repr}
    X(n)\coloneqq X(0)+\sum_{k=0}^{n-1}\left(\xi_{k+1}{\mathds{1}}_{X(n)\notin A} + \sum_{x\in A} \eta_{k+1,x}{\mathds{1}}_{ X(k)=x}\right), \quad n\geq 1.
\end{equation}

 The main  question of the paper is the following.
 
 {\it 
 Does a scaling limit exist for $\Big(\frac{X(nt)}{a_n}\Big)_{t\geq 0}$ under the same normalization as in \eqref{eq:conv_Skor1}, and if so, what is its form?}

Consider two warm up examples.
\begin{example}\label{ex:1_1}
    Assume that $m=0, \quad A=\{0\}, \quad {\mathop{\rm P}}(\xi=\pm1)=1/2, \quad {\mathop{\rm P}}(\eta_0=1)=1,$ and $X(0)\geq 0.$ Then $X$ is a simple random walk with reflection at 0, that is, its transition probabilities are $p_{i,i\pm1}=1/2, i\neq 0$ and $p_{0,1}=1.$

    It is easy to see that if $X(0)\overset{d}{=}|S_\xi(0)|,$ then
    \[
    (X(n))\overset{d}{=}(|S_\xi(n)|).
    \]
    Hence by continuous mapping theorem the sequence   $\Big(\frac{X(nt)}{\sqrt{n}}\Big)_{t\geq 0}$ converges in distribution to reflected Brownian motion $(|B(t)|)_{t\geq 0} $ started at $0.$ 
\end{example}
\begin{example}\label{ex:1_2}
    Assume that similarly to the previous example $ m=0, {\mathop{\rm P}}(\xi=\pm1)=1/2,$ but  ${\mathop{\rm P}}(\eta_0=1)=p $ and ${\mathop{\rm P}}(\eta_0=-1)=q=(1-p), $ where $p\in[0,1].$

Using André's reflection principle it can be seen that $n$-step transition probabilities  for $X$ are
\begin{equation}
    \label{eq:trans_skewRW}
    {\mathop{\rm P}}(X(n)=y | X(0)=x)=\varphi^{(n)}(y-x)+ \gamma \mathop{\rm sgn}(y-x) \varphi^{(n)}(|y|+|x|), \quad x,y\in{\mathbb Z},
\end{equation}
where $\gamma\coloneqq p-q\in[-1,1]$ and  $\varphi^{(n)}(y-x)\coloneqq  {\mathop{\rm P}}(S_\xi(n)=y | S_\xi(0)=x)$ is the transition probabilities for unperturbed random walk.

The modulus of continuity of $X$ is dominated by the modulus of continuity of $S_\xi$ in the following sense:
\[
\forall k,l\geq 1\quad \max_{k\leq i, j\leq l} |X(j)-X(i)| \leq 1+ 2 \max_{k\leq i, j\leq l} |S_\xi(j)-S_\xi(i)|.
\]
Hence, the sequence $\{\Big(\frac{X(nt)}{\sqrt{n}}\Big)_{t\geq 0}, n\in{\mathbb N}\}$ is weakly relatively compact in ${\mathcal D}$ and any of its limit point is a continuous stochastic process. 

As is well known from the de  Moivre-Laplace   local limit theorem, the scaling limit of transition probabilities of $S_\xi$ is  a Gaussian density 
\begin{equation}
    \label{eq:transitionSkewBM}
\lim_{n\to\infty}\sqrt{n}\varphi^{([nt])}([\sqrt{n}z])=\varphi_t(z)=:\frac{1}{\sqrt{2\pi t}}{\rm e}^{-z^2/2t}, \quad t>0, z\in{\mathbb R}.
\end{equation}

Hence, after some (maybe not so trivial) calculations, we see that f.d.d. of  $\Big(\frac{X(nt)}{\sqrt{n}}\Big)_{t\geq 0}$ converge in distribution to f.d.d. of a skew Brownian motion with a parameter $\gamma$, which is a Markov process with transition probability density function being equal to 
\begin{equation}\label{eq:transitionsof skew BM}
     p_t(x,y)=\varphi_t(y-x)+ \gamma \mathop{\rm sgn}(y-x) \varphi_t(|y|+|x|), \quad x,y\in{\mathbb R}.
\end{equation}

Weak relative compactness and convergence of f.d.d. gives us convergence in distribution of $\Big(\frac{X(nt)}{\sqrt{n}}\Big)_{t\geq 0}$  in ${\mathcal D}$ to the skew Brownian motion that starts from 0.
\end{example}

Let us reflect on the previous examples. It is natural to expect that a scaling limit of  $\Big(\frac{X(nt)}{a_n}\Big)_{t\geq 0}$    (if it exists) is a   Markov process that behaves as the process $U$ from \eqref{eq:conv_Skor1} up to the moment of hitting 0, and in particular as a Brownian if the $\xi$ is a centered random variable with finite variance.   Before examining convergence of $\Big(\frac{X(nt)}{a_n}\Big)_{t\geq 0}$, let us answer the following questions.
\begin{itemize}
    \item 
What are non-degenerate examples of processes $U$ from \eqref{eq:conv_Skor1}?
\item Does $U$ visit zero with positive probability?
\item How can we describe all strong Markov processes that behave like $U$ until hitting  0? What processes may appear in a limit? What methods of proof convergence in distribution are helpful?
\end{itemize}  
The answer on the first question is  well known, $U $ must be  strictly $\alpha$-stable processes with some $\alpha\in (0,2]$. If $\alpha\in(1,2],$ then this processes hits 0 with probability 1 and the point 0 is polar if $\alpha\in(0,1),$ see \cite[Example 43.22]{Sato}.
Further we will consider only the case $\alpha\in (1,2]$. It is well known that in this case ${\mathop{\rm E}} \xi=0$, see the classical book \cite{GnedenkoKolmogorov}. Moreover,   $(a_n)$ is regularly varying sequence at infinity of index $1/\alpha$.

The answer on the third question is complicated. In this regard, let us first ask ourselves the question: what are the methods for defining Markov processes, and what methods can be used to prove the convergence of Markov processes? Naturally, this topic is too broad, and we are by no means going to provide exhaustive answers, e.g.,  \cite{Ethier+Kurtz:1986, Dynkin63, ItoMcKean65, ItoSynthesis, Billingsley, gikhmanSkorokhod_vol2, Whitt}.
    We just sketch some ideas and methods that may be useful. At the end of this section, we will provide a  stochastic and analytic description of the skew Brownian motion and the Walsh process, which will be limit processes in most interesting  situations.
 
\textit{Classical approach.} The classical way to determine a stochastic process is via its finite-dimen\-sio\-nal distributions. For homogeneous Markov processes this reduces to knowing the initial distribution and transition probabilities. For example, for the skew Brownian motion transition probabilities are given by formula \eqref{eq:trans_skewRW}, and reflected BM is a particular case of the skew Brownian motion  with $p=1.$  The classical  method  to prove convergence in distribution is through ``tightness and convergence of finite-dimensional distributions'', see e.g. \cite{Billingsley}. In our case a  metric space is the Skorokhod space of c\`adl\`ag functions ${\mathcal D}$ with $J_1$ topology. Certainly, we suggest that we have already known Donsker's theorem or other nice properties of unperturbed  random walks $S_\xi, S_{\eta_x}$ and, usually, a verification of tightness would not be  a big problem, cf. Example \ref{ex:1_2}. However, this traditional approach only works easily in the simplest cases.      If, in Example \ref{ex:1_2}, we use a different distribution for $\xi$ or $\eta$, or a membrane consisting of several points, it will be difficult to find satisfactory explicit formulas for the transition probabilities. Certainly, they can be obtained from some renewal-type equations or like  series  over excursions between consecutive visits the membrane. This way of proof is possible  but nontrivial at all, see \cite{NgoPeigne, Ngo+Peigne:2021}. We should remark that even the case considered in Example \ref{ex:1_2} requires some arguments besides the central limit theorem or de Moivre-Laplace theorem (we omitted a proof of this part there). 

\textit{Analytic methods: Semigroups, generators, resolvents.}  Assume that we  already have a  (time-homogeneous) Markov process $Y$ on the real line with transition probabilities $P(t, x, A), t\geq 0, x\in{\mathbb R}, A\in {\mathcal B}({\mathbb R}).$ Then we can define a semigroup of operators associated with $Y$ that is given by
\[
T_tf(x)\coloneqq {\mathop{\rm E}}_x f(Y(t)), t\geq 0, x\in {\mathbb R},
\]
 acting on a suitable  Banach space of functions $f.$ Further we will always assume that $T_t, t\geq 0, $ is a Feller semigroup, i.e., $T_t$ is a strongly continuous semigroup in the space $C_0({\mathbb R})$ of continuous functions with zero limits at the infinity or in the space $C([-\infty, \infty])$ of continuous functions that have finite limits at infinities. In any case, it is well known that Feller processes possess the strong Markov property  and have  c\`adl\`ag modifications. 
 Once we have formed a semi-group, we can employ powerful analytical methods. Recall at first, that a semigroup is uniquely determined by its generator $A$,
 \[
 Af\coloneqq \lim_{t\to0+}\frac{T_tf-f}{t}
 \]
where the domain $D(A)$ consists of functions $f$ such that the limit exists (in the corresponding Banach space). It is also well known that the resolvent of the semigroup $R_\lambda\coloneqq (\lambda-A)^{-1}, \lambda>0$  also determines the semigroup uniquely and satisfies the equation
\[
R_\lambda f(x)=\int_0^\infty {\mathop{\rm E}}_x f(Y(t)) \mathrm{d} t.
\]

%There is a one-to-one correspondence between Feller semigroups, their generators, and resolvents. The answer on the question if a certain operator is a generator of a Feller semigroup is given by the Hille-Yosida theorem.
Important   fact is that convergence in some sense of resolvents, semigroups, generators, and convergence in distribution of Feller processes are equivalent \cite[ Theorem 2.11 on p. 172]{Ethier+Kurtz:1986}, \cite[Theorem 17.25]{Kallenberg}, and \cite[Theorem 17.27] {Kallenberg} for approximation of Feller processes by Markov chains. Section \ref{sec:skew Levy} will demonstrate the effectiveness of the resolvent approach when both random variables $\xi$ and $\eta$ from the definition of the perturbed RW exhibit heavy-tailed distributions.

\textit{Stochastic methods. } 
The analytic description of via semigroups, resolvents, etc. provides us with important numerical characteristics of stochastic processes and efficient way to prove convergence. On the other hand, describing a stochastic process  using stochastic differential equation, martingale problem, Itô's excursion theory, etc., 
gives us understanding the dynamics of the process. The unperturbed RW and/or Brownian motion may be   ``building blocks'' for a construction  of the perturbed process.
If the perturbed RW and its limit   are continuous functions in some sense of this blocks, then the proof would follows from the combination of   Donsker's type theorem together with  continuous mapping theorem. We will see that this approach is helpful for reflected processes (see \S \ref{sec:reflection}), and  what is even more surprising, for permeable membrane, where the limit process is a skew Brownian motion. %, despite the fact that local time is not a continuous function of trajectory.  

Each  of described  method has its own advantages and it is difficult to say apriori which method of investigation is  more effective.

 We finish this section by analytic and stochastic description of   a skew Brownian motion and its natural generalization the Walsh process, see  \cite{Walsh:1978, BarlowPitmanYor, IPMS_book, Bayraktar, Lejay:skew, Bobrowski-Pilipenko_2025_Walsh, PavlyukevichPilipenko2023} for references and further properties.
 
 \begin{example}
     Let $X$ be a skew Brownian motion with parameter $\gamma\in[-1,1]$, i.e., a continuous Markov process with transition probability density function given in \eqref{eq:transitionsof skew BM}. %(see  the various descriptions in the review article \cite{Lejay:skew}). %It is well known that  $X$ is a Feller process. 
     
     \textit{Analytic description.} The domain of the generator $A: Dom(A)\subset C([-\infty,\infty])\to C([-\infty,\infty])$ consists of all functions $f\in C([-\infty,\infty])$ such that 
     
     (a) the second derivative $ f''(x) $ exists for all $ x\neq 0$ and can be extended by continuity to all $x\in[-\infty, \infty];$
     
     (b) the left and right derivatives  $f'(0\pm)$  at 0 exist  and satisfy the conjugation condition at 0  
     \begin{equation*}
(1-\gamma) f'(0-)=(1+\gamma) f'(0+).
     \end{equation*}
  If $f\in Dom(A),$ then 
  $A f(x)=      \frac12 f''(x), x\in{\mathbb R}$, where $f''(0)\coloneqq 
      f''(0+)=f''(0-).
  $

The resolvent is of the form
\[
R^{{\rm skew}}_\lambda f(x) =\int_{\mathbb R} r_\lambda^{{\rm skew}}(x,y) f(y) {\rm d}y,
\]
where
\[
r_\lambda^{{\rm skew}}(x,y)=\frac{1}{\sqrt{2\lambda}}({\rm e}^{-\sqrt{2\lambda}|x-y|}+
 \gamma \mathop{\rm sgn} (y) {\rm e}^{-\sqrt{2\lambda}(|x|+|y|)}),\quad x,y\in{\mathbb R}.
\]

  \textit{Stochastic description.} It\^o and McKean obtained a skew Brownian motion from a reflected Brownian motion (both started from 0, for simplicity) via the following procedure, see    \cite[Problem 1 of Section 4.2]{ItoMcKean65}. They enumerated excursions of the reflected Brownian motion in any measurable way and changed the sign of each excursion independently with probability $1-p$  and  keeping it positive with probability $p$.  Harrison and Shepp \cite{Harrison+Shepp} proved that for any $\gamma\in[-1,1]$ there is a unique strong solution to the SDE
  \[
  \mathrm{d} X(t)=\mathrm{d} W(t) +\gamma \mathrm{d} L_0^X(t),
  \]
where $W$ is a Brownian motion and $ L_0^X $ is a symmetric local time of $X$ at 0, and the solution is the skew  Brownian motion with parameter $\gamma.$  We postpone another stochastic representation until the next example, where more general process is considered.
 \end{example}

\begin{example}[The Walsh process]

The Walsh process is a natural generalization of the skew Brownian motion. This is a continuous, homogeneous  Markov process whose state space is a star graph consisting   $d$ rays with a common endpoint. Similarly to the It\^o -- McKean construction of the skew Brownian motion,  Walsh  \cite{Walsh:1978} proposed distributing the excursions of reflected Brownian motion along these rays so that the excursion would fall on the $ k$th ray with probability $p_k,$  where $p_1,\ldots,p_d$ are fixed non-negative numbers that are summed up to 1. 
Without loss of generality we will assume that the state space $E_d$ for the Walsh Brownian motion is the union of nonnegative coordinate semi-axes in ${\mathbb R}^d$, that is,
\[
E_d\coloneqq\{x\in{\mathbb R}^d\ :\  x_i\geq 0\text{ and } x_ix_j=0,\ i\neq j,\  i,j=1,\dots, d\}.
\]
   \textit{Analytic description.} Transition probability density  of the Walsh process is given by
\begin{align*}
p_t(x,y) 
=
\begin{cases}
\varphi_t (x_i-y_j)
 + (2p_i-1)\varphi_t (x_i+y_j),& \
i=j \\
2p_j\varphi_t (x_i+y_j),& \
i\neq j
\end{cases}
\end{align*}
  where $x=(0,\ldots,0,x_i,0,\ldots,0), y=(0,\ldots,0,y_j,0,\ldots,0)\in E$, $x_i\geq 0, y_j \geq 0, $ and $\varphi_t(z)=(2\pi t)^{-1/2} {\rm e}^{-z^2/2t},\quad  z\in{\mathbb R}, t>0$ is the Gaussian density.

It follows from Walsh's definition or from the form of transition probabilities that 
a Walsh process $(X_1,\dots,X_d)$ with parameters $p_1,\dots,p_d$ satisfies the following: 
for every $K\subset \{1,\dots, d\}$ the one-dimensional process $\sum_{i\in K}X_i(t)-\sum_{j\notin K}X_j(t)$ 
is the skew Brownian motion with parameter $\gamma\coloneqq \sum_{i\in K}p_i -\sum_{j\notin K}p_j.$ 
In particular, $\sum_{j=1}^d X_j(t)$ is a reflected Brownian motion.

  % It follows from Walsh's definition or from the form of transition probabilities that if $(X_1, \dots, X_d)$ is a Walsh process with parameters $p_1,\dots,p_d,$ then for every $K\subset \{1,\dots, d\}$ the one-dimensional process  $\sum_{i\in K}X_i(t)-\sum_{j\notin K}X_j(t)$ is the skew Brownian motion with parameter $\gamma\coloneqq \sum_{i\in K}p_i -\sum_{j\notin K}p_j.$ In particular, $\sum_{j=1}^d X_j(t)$ is a reflected Brownian motion.

We identify the space of continuous functions $f$ on compactified graph with a subspace of  $(C([0,\infty]))^d$ of functions $(f_1,\dots,f_d)$ such that
$f_i(x_i)= f(0,\dots,0,x_i,0,\dots,0).$ Note that $f_i(0)=f_j(0), 1\leq i,j\leq d.$

The domain of generator consists of all $(f_1,\dots,f_d)\in   (C^2( [0,\infty]))^d$  such that $f_i''(0+)=f_j''(0+), 1\leq i,j\leq d,$ and $\sum_{i=1}^d p_if'_i(0+)=0.$ The generator is equal to a half of the second derivative along the coordinate axes. See comprehensive description of Walsh process  and its generalizations in  \cite{Bobrowski-Pilipenko_2025_Walsh}.

   \textit{Stochastic description.}
 The following martingale problem for the Walsh process was introduced by Barlow, Pitman, and Yor in \cite{BarlowPitmanYor}. Here, we adopt the formulation from \cite{PavlyukevichPilipenko2022}.
\begin{theorem} 
\label{thm:Walsh} 
Let $X= (X_1,\dots,X_d)$ and $L$ be two continuous processes, the first with values in 
${\mathbb R}^d$, the second with values in ${\mathbb R}$, and  let  $\mathcal{F}^X$ be the filtration generated by $X$. Then $X$ is a Walsh's process with parameters $p_1,\dots,p_d$, and $L$ is its local time at   $0$ defined by 
\[
L(t)\coloneqq \lim_{\varepsilon\to0+}\tfrac{1}{2\varepsilon}\int_0^t{\mathds{1}}_{\{0<|X(s)|\leq \varepsilon\}}\mathrm{d} s, \quad t \ge 0,
\]
if and only if
\begin{itemize}
\item [(1) ] $X_i(t)\geq 0, 1\leq i\leq d $, and $X_i(t)X_j(t)=0, 1\leq i,j\leq d, i \neq  j$ for $t\geq 0$;
\item [(2) ] $L$ is a.s.\ nondecreasing $\mathcal{F}^X$-adapted process with $L(0)=0$ and
\[ 
\int_{[0,\,\infty)}{\mathds{1}}_{\{X(s)\neq 0 \}}\, {\rm d} L(s)=0\quad\text{a.s.};
\]
\item [(3) ] the processes $M_i, 1\leq i\leq d $ defined by
\begin{equation}
\label{e:Mnu}
M_i(t)\coloneqq X_i(t)-p_i L(t),\quad t\geq 0
\end{equation}
are continuous square integrable martingales with respect to $\mathcal{F}^X$ with predictable quadratic variations
\begin{equation}
\label{e:brM}
\langle M_i\rangle (t)=\int_0^t {\mathds{1}}_{\{X_i(s)>0\}} \mathrm{d} s, \qquad t \ge 0;
\end{equation}
\item [(4) ] $\int_0^\infty {\mathds{1}}_{\{X(s)=0\}} \mathrm{d} s=0$ {\rm a.s.}
\end{itemize}
\end{theorem}

 Another stochastic description, which will be used in this paper, is taken from \cite{Bayraktar}, see also \cite{Bobrowski-Pilipenko_2025_Walsh}.
   \begin{theorem} 
    \label{thm:baryaktar} Let $p_1,\dots,p_d$ be   positive numbers such that $\sum_{1\leq i\leq d} p_i=1$. Given
 a standard $d$-dimensional Wiener process  $ W=(W_1,\dots,W_d)$ with $W(0)\in E_d$, we define 
\[ L_i(t)\coloneqq -\min_{s\in[0,t]}(W_i(s)\wedge 0) \text{ and } W^{\text{refl}}_i(t)\coloneqq W_i(t)+L_i(t), \quad 1\leq i\leq d, t \ge 0.\]     
    Then there is a unique collection of continuous, non-negative and non-de\-creas\-ing processes $T_i, 1\leq i\leq d$ such that almost surely
    \begin{equation}
        \label{eq:balance_locTime_Walsh's process1}        \frac{L_i(T_i(t))}{p_i}=\frac{L_j(T_j(t))}{p_j},\qquad  1\leq i,j\leq d,\quad t\geq 0, 
    \end{equation}
and $\sum_{1\leq i,j\leq d}T_i(t)=t, t\geq 0$.
    Moreover, the process 
    \begin{equation}
                \label{eq:repr_Bayr}
        X\coloneqq ({W^{\text{refl}}_i\circ T_i} )_{1\leq i\leq d}
        \end{equation}
is then Walsh's process  with parameters $p_1,\dots,p_d.$ 
    \end{theorem}

\end{example}

\section{First steps.  }\label{sec:firststeps}

In this section we give answers to the following general questions:
\begin{itemize}

 \item 
    Assume that we have already proven that there exists a scaling limit of sequence of perturbed RW $\Big(\frac{X(nt)}{a_n}\Big)_{t\geq 0}\Rightarrow  \Big(X_\infty(t)\Big)_{t\geq 0}. $   Consider a continuous time RW $(X^{cont}(t))$ that has the same jumps as $(X(k))$ but spends time $\varepsilon_k$ at the state $X(k)$ between $k$th and $(k+1)$st jumps, where $(\varepsilon_k)$ are i.i.d.  positive random variables with finite expectation. In particular, if $(\varepsilon_k)$ are exponentially distributed, then $X^{cont}$ is a continuous time Markov chain.
What is a scaling limit for   $X^{cont}$? 

The answer is ``no changes'', see Lemma \ref{lem:subordinating} below.

    \item How often $X$ visit the membrane $A?$
    
  Answer: Under natural assumptions, after $n$ steps, the number of visits of $X$ to $A$ is   $o(n)$  as $ n\to\infty,$  see Lemma \ref{lem:neglidg_critical}.

    \item   Assume that we have already proven that there exists a scaling limit of sequence of perturbed RW $\Big(\frac{X(nt)}{a_n}\Big)_{t\geq 0}\Rightarrow  \Big(X_\infty(t)\Big)_{t\geq 0}. $ Since $X$ visits the membrane rarely, can we without loss of generality assume that $X$ exits $A$ at next step after the entrance? Would it affect the  scaling limit?

  Answer: Under natural assumptions, without loss of generality we may assume that $${\mathop{\rm P}}_x(X(1)\notin A)=1, x\in A,$$  see Lemma \ref{lem:equivalent perturbed with skips}.

\item  What is the equation or stochastic representation for the limit process?

Informal arguments are given after  Lemma \ref{lem:equivalent perturbed with skips}, see formula \eqref{eq:guessLimit}. In subsequent sections, this formula and its modification will repeatedly appear in a variety of situations.
   
\end{itemize}

Before presenting the results, let us recall Skorokhod's representation theorem, which enables many proofs to focus on almost sure convergence instead of convergence in distribution. This important result will be used frequently throughout the paper.
\begin{theorem}[The Skorokhod representation theorem]\label{thm:SkorokhodRepresentation}
A sequence  $(\theta_n) $   of random elements taking values in a separable metric space $E$  converges in distribution
\[
\theta_n~\Rightarrow~\theta_\infty,\quad n\to\infty.
\]
if and only if there exists a probability space, a sequence
$(\tilde \theta_n) $ and a random element $\tilde \theta_\infty$ defined on this probability space such that $\tilde \theta_n \overset{{\rm d}}= \theta_n$ for all $ n\geq 1$, $\tilde \theta_\infty \overset{{\rm d}}= \theta_\infty$, and
\[
\lim_{n\to\infty} \tilde \theta_n~=~\tilde\theta_\infty 
\quad \text{{\rm a.s.}}
\]
\end{theorem}
 See the proof in~\cite{Skorokhod_Issl} or see Theorem 3.30 in~\cite{Kallenberg}.

\begin{lemma}\label{lem:subordinating}
Let $(\varepsilon_n)$ be a sequence of positive random variables such that $\lim_{n\to\infty} \frac{\varepsilon_1 + \dots + \varepsilon_n}{n} = c$ a.s., where $c > 0$ is a non-random constant. Define the counting process $N(t) \coloneqq n$ for $t \in [\sum_{i=1}^n \varepsilon_i, \sum_{i=1}^{n+1} \varepsilon_i)$, $n \geq 0$, where $\sum_{i=1}^0 \coloneqq 0$. Set $X^{cont}(t) \coloneqq X(N(t))$ for $t \geq 0$.  

We have convergence in distribution in ${\mathcal D}$:
\begin{equation}
    \label{eq:conv_PRW}
    \Big(\frac{X(nt)}{a_n} \Big)_{t \geq 0} \Rightarrow \Big(X_\infty(t) \Big)_{t \geq 0}, \quad n \to \infty,
\end{equation}
if and only if
\[
\Big(\frac{X^{cont}(nt)}{a_n} \Big)_{t \geq 0} \Rightarrow \Big(X_\infty\Big(\frac{t}{c}\Big) \Big)_{t \geq 0}\quad n \to \infty.
\]
\end{lemma}

\begin{proof}
For a non-decreasing function $f:{\mathbb R}_+\to{\mathbb R}_+$, we define the generalized inverse by
\[
f^{-1}(t)\coloneqq \inf\{s\geq 0\ : \ f(s)>t\},
\]
see \cite{Whitt} for a detailed investigation of its properties.

Note that $N(t)=S_\varepsilon^{-1}(t)-1$, where $S_\varepsilon(n)=\varepsilon_1+\dots+\varepsilon_n$ for $n\geq 1$ and $S_\varepsilon(0)=0$. The almost sure convergence $\lim_{n\to\infty}\frac{S_\varepsilon(n)}{n}=c$ implies the locally uniform convergence
\[
\forall T>0\ \ \lim_{n\to\infty}\max_{t\in[0,T]}\Big|\frac{S_\varepsilon(nt)}{n}-c t\Big|=0 \quad \text{a.s.}
\]
Since the function $f(t)=ct$ is strictly increasing, we have a.s. locally uniform convergence of the inverses (see \cite[Corollary 13.6.4]{Whitt}):
\[
\forall T>0\ \ \lim_{n\to\infty} \max_{t\in[0,T]}\Big|\frac{S^{-1}_\varepsilon(nt)}{n}- \frac{t}{c} \Big|=0.
\]
The limit in the r.h.s. is non-random; therefore, we have the joint convergence of pairs \cite[Theorem 3.9]{Billingsley}:
\[
\Big(\frac{X(nt)}{a_n}, \frac{N(nt)}{n}\Big)_{t\geq 0}\Rightarrow  \Big(X_\infty(t), \frac{t}{c} \Big)_{t\geq 0}.
\]
It follows from the Skorokhod representation theorem that there exist copies $\left(\frac{\widetilde X_n(nt)}{a_n}, \frac{\widetilde N_{n}(nt)}{n}\right)$ of $\left(\frac{X(nt)}{a_n}, \frac{N(nt)}{n}\right)$, defined on a common probability space, that converge a.s. in ${\mathcal D}$ to $\left(\widetilde X_\infty(t), \frac{t}{c}\right)$. Here the first coordinate has the same distribution as $X_\infty$. Thus, we have a.s. convergence of compositions in ${\mathcal D}$ (see \cite[Theorem 13.2.1]{Whitt}):
\[
 \Big( \frac{\widetilde X_n(n\cdot)}{a_n}\circ \frac{\widetilde N_{n}(nt)}{n}\Big)_{t\geq 0}  =  \Big(\frac{\widetilde X_n({\widetilde N_{n}(nt)})}{a_n}\Big)_{t\geq 0}\to  \Big(\widetilde X_\infty\Big(\frac{t}{c}\Big) \Big)_{t\geq 0}, \ n\to\infty,
\]
because $g(t)=\frac{t}{c}$ is a continuous and strictly increasing function. This proves necessity. Sufficiency may be proven similarly.
\end{proof}

%%%%%%%%%%%%%%%%%%%%%%%%%%%%%%%%%%%%

It would be natural to use unperturbed random walk $S_\xi$ and its properties in a construction of the perturbed random walk. However, some  random variables  among $ \xi_1,\dots,\xi_n$ are not used  in representation  \eqref{eq:transitions_PRW}: we skip those $\xi_k$s when the perturbed random walk visits the membrane. Below we   present another  than \eqref{eq:X_repr} probabilistic representation  of Markov chain with transition probabilities \eqref{eq:transitions_PRW}. 
 
Let random variables  $(\xi_n), (\eta_{n,x}), x\in A$, and perturbed RW $X$ be as in \S \ref{sec:2formulation}.  Set $Y(0)\coloneqq X(0)$  and construct   $(Y(n))_{n\geq 1}$ recursively.
Assume that the random variables $Y(0), Y(1), \dots, Y(n)$ are known.

Let $T^{(x)}(n)$, for $x \in A$, denote the number of visits to $x$ up to time $n$: 
\[
T^{(x)}(n)\coloneqq \sum_{k=0}^{n-1} {\mathds{1}}_{\{Y(k)=x\}},\quad x\in A,%|x|\leq m.
\]
where, as before, we define $\sum_{k=0}^{-1} \coloneqq 0$. Furthermore, let $T^{\text{normal}}(n)$ and $T^{\text{critical}}(n)$ represent the total time spent outside the membrane and at the membrane, respectively:
\begin{align*}
T^{\text{normal}}(n)&\coloneqq n-\sum_{x\in A}T^{(x)}(n)= \sum_{k=0}^{n-1}{\mathds{1}}_{\{Y(k)\notin A\}} \\ T^{\text{critical}}(n)&\coloneqq  \sum_{x\in A}T^{(x)}(n)= \sum_{k=0}^{n-1}{\mathds{1}}_{\{Y(k)\in A\}}.
\end{align*}
Set $Y(n+1)\coloneqq Y(n)+\xi_{T^{\text{normal}}(n)+1}$ if $Y(n)\notin A$ and $Y(n+1)\coloneqq Y(n)+\eta^{(x)}_{T^{(x)}(n)+1}$ if $Y(n)=x\in A$. So,
 \begin{equation}\begin{aligned}\label{eq:20_representation_PRW}
Y(n)&= Y(0)+\sum_{k=1}^{T^{\text{normal}}(n)} \xi_k +\sum_{x\in A}
\sum_{k=1}^{T^{(x)}(n)}\eta^{(x)}_{k} \\
&= S_\xi( T^{\text{normal}}(n))+\sum_{x\in A} S_{\eta^{(x)}}(T^{(x)}(n)), \quad n\geq 0,
\end{aligned}\end{equation}
where $S_\xi(n)=Y(0)+\sum_{k=1}^n\xi_k, $ and $S_{\eta^{(x)}}(n) = \sum_{k=1}^n \eta^{(x)}_{k}$.

It can be seen that $(X(n))_{n\geq 0}\overset{d}{=}(Y(n))_{n\geq 0}.$ Representation \eqref{eq:20_representation_PRW} is often more convenient for  the intuition and proofs. 

Further we consider only the case when $U$ in \eqref{eq:conv_Skor1} is non-degenerate $\alpha$-stable process with $\alpha\in(1,2].$ The case  $\alpha \in(0,1) $ and finite $A$  is less interesting from the probabilistic point of view.  Indeed, in this case  $S_\xi$  is transient, and hence (after some justification) the chain $Y$. The latter implies  that the limit   $T^{\text{critical}}(+\infty)\coloneqq \lim_{n\to\infty} T^{\text{critical}}(n)$ is finite a.s. Hence, the limit of $\Big(\frac{Y(nt)}{a_n}\Big)_{t\geq 0}$ is simply $(U(t))_{t\geq 0}$, coinciding with the limit of  $\Big(\frac{S_\xi(nt)}{a_n}\Big)_{t\geq 0}$.

The next result demonstrates that, under natural assumptions, the perturbed random walk  visits the membrane relatively infrequently.
 \begin{lemma}
    \label{lem:neglidg_critical} 
    Let $Y$ be a perturbed random walk with the representation \eqref{eq:20_representation_PRW}, where ${\mathop{\rm E}} \xi=0$. Assume that all states of ${\mathbb Z} \setminus A$ communicate and that the random walk cannot stay in the membrane forever with probability 1. Furthermore, assume that either:
    \begin{itemize}
        \item the membrane $A$ is finite;
        \item or $A = {\mathbb Z} \setminus {\mathbb N}$ and the random walk $Y$ exits the membrane immediately after any entry, i.e., ${\mathop{\rm P}}(Y(n+1) \notin A \mid Y(n) = x) = 1$ for all $x \in A$.
    \end{itemize}
    Then, with probability 1, the following convergences hold % uniformly on compact sets
    for any $T > 0$:
    \begin{align*}
        \lim_{n\to\infty} \sup_{t \in [0,T]} \frac{T^{\text{critical}}(nt)}{n} = 0, \quad \text{and} \quad \lim_{n\to\infty} \sup_{t \in [0,T]} \Big| \frac{T^{\text{normal}}(nt)}{n} - t \Big| = 0.
    \end{align*}
\end{lemma}
\begin{proof}
For $A = {\mathbb Z} \setminus {\mathbb N}$, the Markov chain $Y$ can visit $A$ at time $k$ only if $T^{\text{normal}}(k)$ is a descending ladder epoch of $S_\xi$. Since $T^{\text{normal}}(k) \leq k$, the total number of such visits by $Y$ up to $nT$ is at most the number of descending ladder epochs of $S_\xi$ in $[0, \lfloor nT \rfloor + 1]$. 
The expectation of the descending ladder epoch is infinite because ${\mathop{\rm E}}\xi=0$; cf. \cite[\S 2.9, Theorems 9.1 and 9.2]{Gut:2009} and \cite[Chapter XII.2, Theorem 2]{Feller1966Vol2}. Therefore, \cite[\S 3.4, Theorem 4.1]{Gut:2009} implies $\lim_{n\to\infty} T^{\text{critical}}(n)/n=0$ a.s.  

The rest of the proof follows from the observations that
\[
\sup_{t\in[0,T]} \frac{T^{\text{critical}}(nt)}{n} = \frac{T^{\text{critical}}(nT)}{n}
\]
and 
\[
\Big|\frac{T^{\text{critical}}(nt)}{n} + \frac{T^\text{normal}(nt)}{n} - t\Big| \leq \frac{1}{n}.
\]

%We leave the case when $A$ is finite to the reader.
We give the idea of the proof in the case $A$ is finite. As in the representation \eqref{eq:20_representation_PRW}, jumps outside $A$ can be divided into two types: those that occur when $Y$ is above $A$, and those when $Y$ is below $A$ (see the more general representation \eqref{eq:WalshRW_repr} for a random walk on a graph). The number of times $Y$ enters $A$ from above does not exceed the number of descending ladder epochs of a RW that used jumps when $Y$ is above $A$. The same logic applies to entrances from below. Consequently, the total number of entrances into $A$ is $o(n)$ a.s. It is not necessary to assume that $Y$ leaves the membrane immediately after entering, since $A$ is finite and its expected exit time is also finite.
\end{proof}

%%%%%%%%%%%%%%%%%%%%%%%%%%%%%%%%%%%%%%%%%%%%%%%%%%%%%%%%%%%%%%%%%%%%

%rata{
Let $Y$ be a perturbed random walk with transition probabilities \eqref{eq:transitions_PRW} and representation \eqref{eq:20_representation_PRW}. Define two associated Markov chains, $\widetilde Y$ and $\hat Y$. The sequence $(\widetilde Y(n))$ is derived from $(Y(n))$ by removing all steps where $Y$ is in $A$, while $\hat Y$ is constructed by omitting jumps within $A$. These can be defined as:
\[
\widetilde Y(n) = Y( (T^\text{normal})^{-1}(n)-1)
\]
and
\[
\hat Y(n) = Y((\hat T^\text{normal})^{-1}(n)-1),
\]
where $\hat T^\text{normal}(n) \coloneqq  \sum_ {k=0}^{n-1} \big({\mathds{1}}_{Y(k)\notin A} + {\mathds{1}}_{Y(k)\in A,\, Y(k-1)\notin A}\big)= n-\sum_ {k=0}^{n-1} {\mathds{1}}_{ Y(k-1)\notin A,\, Y(k)\notin A}$.  
 
Set $\sigma_B\coloneqq \sigma_B(Y)\coloneqq \inf\{k\geq 1 \colon Y(k)\in B\}.$
Transition probabilities  of $\widetilde Y, \hat Y$ are
 \begin{equation}
     \begin{aligned}
\label{eq:transitions_PRW skips1}
    \widetilde p_{x,y}&= 
     {\mathop{\rm P}}(Y(\sigma_{{\mathbb Z}\setminus A} )=y \ | \    Y(0)=x) 
     \\
    &={\mathop{\rm P}}(\xi=y-x) + \sum_{u\in A}{\mathop{\rm P}}(\xi=u-x){\mathop{\rm P}}( Y(\sigma_{{\mathbb Z}\setminus A})=y\ | \   Y(0)=u) , & x,y\notin A;
  \end{aligned}
\end{equation}
\begin{equation}
     \begin{aligned}   
    \label{eq:transitions_PRW skips2}
    \hat p_{x,y}&=\begin{cases}
    {\mathop{\rm P}}(\xi=y-x), & x\notin A, y\in{\mathbb Z};\\
    {\mathop{\rm P}}(Y(\sigma_{{\mathbb Z}\setminus A} )=y \ | \    Y(0)=x), & x\in A, y\in {\mathbb Z}.
\end{cases}
\end{aligned}
\end{equation}

 Note that the state space for  $\widetilde Y$ is ${\mathbb Z}\setminus A$. 

\begin{lemma}
    \label{lem:equivalent perturbed with skips}

1) If  $(Y(n\cdot )/a_n)$ converges in distribution in ${\mathcal D}$ and   
\begin{equation}\begin{aligned}\label{eq:critical small}
\frac{T^{\text{critical}}(n )}{n}%\coloneqq \frac{\int_0^{n} {\mathds{1}}_{  Y (s)\in A}\mathrm{d} s}{n}
\overset{{\mathop{\rm P}}}{\to}0, \quad n\to \infty, 
\end{aligned}\end{equation}
then sequences  $(\widetilde Y(n\cdot )/a_n)$ and $(\hat Y(n\cdot )/a_n)$ converge in distribution.

2) Assume that   $(\widetilde Y(n\cdot )/a_n)$  or  $(\hat Y(n\cdot )/a_n)$ converges in distribution in ${\mathcal D}$, \eqref{eq:critical small} holds   and   
 \begin{equation}\begin{aligned}\label{eq:jumps in membrane}
 \max_{0\leq s_1<s_2\leq  t}\frac{|Y(ns_2)-Y(ns_1)|}{a_n}{\mathds{1}}_{   Y(z)\in A, \ z\in (ns_1,ns_2]  }\overset{{\mathop{\rm P}}}{\to} 0, \quad n\to \infty.
 \end{aligned}\end{equation}
 %where $(a_n)$ is a sequence of positive numbers, $\lim_{n\to\infty} a_n=\infty.$
 Then $(Y(n\cdot )/a_n)$  converges in distribution.
 
 In all cases all the limits coincide.
\end{lemma}

\begin{proof}[Proof of Lemma \ref{lem:equivalent perturbed with skips}]
Applying Skorokhod's representation theorem, we may work with copies of the stochastic processes that converge a.s. The proof of the first part of the Lemma is the same as  in Lemma \ref{lem:subordinating}. The second case is a consequence of Skorokhod's representation theorem together with \cite[Proposition 4.3.10]{IPMS_book}, which provides a sufficient condition guaranteeing that the limit of compositions of càdlàg functions is equal to the composition of their  limits.
\end{proof}
\begin{remark}\label{remk:skips}
     Condition \eqref{eq:critical small} is satisfied if, for example, conditions of Lemma \ref{lem:neglidg_critical} are satisfied.
Condition \eqref{eq:jumps in membrane} means that $Y$ cannot enter $A$ by a large jump   and also that $Y$ cannot make too big jumps inside $A$ (the last condition make sense only if $A$ is infinite). Observe that if $A\subset [-m,m],$ then  the left-hand side of \eqref{eq:jumps in membrane} does not exceed $\frac{2m}{a_n}+ \max_{0\leq k\leq nt}\frac{|\xi_{k+1}-\xi_k|}{a_n}$; if $A={\mathbb Z}\setminus {\mathbb N}$ and ${\mathop{\rm P}}(Y(n+1)\in{\mathbb N} | Y(n)\in A)=1$, then the left-hand side of \eqref{eq:jumps in membrane} does not exceed $ \max_{0\leq k\leq nt}\frac{|\xi_{k+1}-\xi_k|}{a_n}.$ 
Condition \eqref{eq:jumps in membrane} is satisfied in previous two cases if ${\mathop{\rm E}} \xi=0, {\mathop{\rm Var}} \xi<\infty, a_n=\sqrt{n}$ because $\max_{0\leq k\leq n}\frac{\xi_k}{\sqrt{n}}\overset{{\mathop{\rm P}}}{\to}0$, e.g., \cite[Theorem 2.1 and Remark 2.1, p. 270]{GutGraduate}. It can be shown (but not so straightforward), that   \eqref{eq:jumps in membrane} is also true if $A$ is finite, $\xi$ belongs to the domain of attraction of symmetric $\alpha$-stable law with $\alpha\in(1,2)$ and $(a_n)$ is the corresponding scaling sequence.
The proof follows from the fact that with probability 1 symmetric $\alpha$-stable process cannot enter 0 by a jump. Consequently, prelimit sequence also cannot enter the membrane by a large jump with high probability.
\end{remark}

%%%%%%%%%%%%%%%%%%%%%%%%%%%%%%%%%%%%%%%%%%%%%%%%%%%%%%%%%%%%%%%%%%%%%%%%%%

To prove existence of a scaling limit for perturbed random walks, it is helpful to have some idea of what the limit or its characteristics might be. Next, we propose a conjecture about the stochastic description of the limit process.

Under conditions of  Lemma \ref{lem:neglidg_critical} and \eqref{eq:conv_Skor1} we have convergence in distribution (see   arguments based on Skorokhod's representation theorem in the proof of Lemma \ref{lem:subordinating}):
\begin{equation}\label{eq:conv-stab-proc}
    \Big(\frac{S_\xi(T^{\text{normal}}(nt)  )}{a_n}\Big)_{t\geq 0}=    \Big(\frac{S_\xi(n\cdot  )}{a_n}\circ \frac{T^{\text{normal}}(nt)}{n}\Big)_{t\geq 0}\Rightarrow (U(t))_{t\geq 0}, \quad n\to\infty,
\end{equation}
where $U$ is an $\alpha$-stable Lévy process.

If we expect existence of scaling limits of $Y$ it is reasonable to assume that there are scaling limits of  $S_{\eta^x}, x\in A$ (but maybe with other scaling than for $S_\xi $). 
\[
 \Big(\frac{S_{\eta_x}(nt)}{a_{n,x}}\Big)_{t\geq 0}\Rightarrow \Big(V_x(t)\Big)_{t\geq 0}, \quad n\to\infty.
\]
After some quite standard manipulations with   scaling sequences, which are regularly varying at infinity, we can find sequences $(b_{n,x})_{n\in{\mathbb N}},   x\in A,$ such that
\[
 \Big(\frac{S_{\eta_x}(b_{n,x}t)}{a_{n}}\Big)_{t\geq 0}\Rightarrow \Big(V_x(t)\Big)_{t\geq 0}, \quad n\to\infty,
\]
for example, we may take $b_{n,x}:= \inf\{k\geq 1 : a_{k,x}>a_n\}.$

Hence, we obtain 
\[
\frac{Y(nt)}{a_{n}}=  \frac{S_\xi(  {T^{\text{normal}}(nt) } )}{a_n}+\sum_{x\in A} \frac{S_{\eta^{(x)}}({b_{n,x}}\frac{T^{(x)}(nt)}{b_{n,x}})}{a_n}.
\]

Denote by $\tau_n$ the instant of the $n$-th visit of $Y$ to $A$. Since $Y$ is recurrent by our assumptions, $\tau_n < \infty$ a.s., and the sequence $(Y(\tau_n))$ is a Markov chain on $A$. 

Recall our assumption that all states of ${\mathbb Z} \setminus A$ communicate for the chain $(Y(n))$ and that $Y$ exits $A$ with probability $1$. Consequently, if $A$ is finite, there exists a unique stationary distribution $(\pi_x, x \in A)$ for the Markov chain $(Y(\tau_n))$. By the law of large numbers for Markov chains, we have the a.s. limit of ratios:
\begin{equation}
    \label{eq:proport_times}
\lim_{n\to\infty}\frac{T^{(x)}(n)}{T^{(y)}(n)}=
\lim_{n\to\infty}\frac{\sum_{k=1}^n{\mathds{1}}_{\{Y(\tau_k)=x\}}}{\sum_{k=1}^n{\mathds{1}}_{\{Y(\tau_k)=y\}}} = \frac{\pi_x}{\pi_y}, \quad x,y \in A. 
\end{equation}

It follows from \eqref{eq:proport_times} that if $\left(\frac{T^{(x)}(nt)}{b_{n,x}}\right)_{t \geq 0}$ converges to a non-degenerate process $(L(t))_{t \geq 0}$, then $\left(\frac{T^{(y)}(nt)}{b_{n,x}}\right)_{t \geq 0}$ converges to $\left(\frac{\pi_y}{\pi_x}L(t)\right)_{t \geq 0}$. 

The processes $V_x, x \in A,$ must be strictly stable Lévy processes; let their corresponding indices be denoted by $\beta_x$. It is well known that if $\beta_x < \beta_y$, then $b_{n,y} = o(b_{n,x})$ as $n \to \infty$. Furthermore, a linear combination of independent stable processes with the same index $\beta$ is itself a $\beta$-stable process. Therefore, our hypothesis is that the limit $(\frac{Y(nt)}{a_n})_{t \geq 0}$ takes the form:
\begin{equation}
    \label{eq:guessLimit}
    Z(t) = U(t) + V(L(t)),
\end{equation} 
where $V$ is a stable Lévy process with index $\beta  = \min_{x \in A} \beta_x$.

 The process $L$ is non-decreasing and can only increase when $Z$ visits 0, since $T^{(x)}$ is non-decreasing and can only increase when $Y$ visits $A$. 
The times between successive visits $(Y(n))$ of any fixed $x\in A$ are positive independent and identically distributed random variables, and $L$ is considered the scaling limit of the number of visits to some $x^*\in A$. Therefore,  $L$ is an inverse subordinator and hence is continuous in $t\geq0.$ Thus, we may expect that $L$ is a continuous additive functional of $Z$ at 0, i.e., a local time of $Z$ at 0. This question is delicate and needs thorough considerations. There are a lot of definitions of a local time: symmetric, right-continuous, semimartingale, Blumenthal-Getoor local time, etc. Usually, all of them are equivalent up to a multiplicative constant, but the choice of  the constant  is not intuitive at all. Generally, the local time of a Markov process is not a continuous function of the process. Convergence towards a local time  of scaling   the number of visits of a random walk to a fixed point is 
  not trivial, even for a simple,  unperturbed RW $S_\xi$ with ${\mathop{\rm P}}(\xi=\pm1)=1/2.$ The general case  needs advance methods, cf. \cite{Borodin_1985_local_time}.

Another crucial consideration involves the following question:   For which  $\beta\in(0,2]$  equation \eqref{eq:guessLimit} has a solution?   Answer: not for all. An informal  explanation is the following. It is well known from results of Feller \cite{feller1957boundary} and Wentzell \cite{Wentcel1956boundary} that the boundary condition at 0 for one-dimensional Brownian motion on $[0,\infty)$ includes a measure $m$ such that $\int_0^\infty (1\wedge x)\; m(\mathrm{d} x)<\infty.$ This gives an idea that the process $V$ should have bounded variation and hence $\beta\in(0,1) $ or $\beta=1$ and $V$ is a linear function. 
In fact, this is true if $\alpha=2$ and $U$ is a Brownian motion (with an appropriate choice of the local time $L$ of $Z$),   see, e.g., \cite{Bobrowski-Pilipenko_2025_Walsh}.  Moreover, if $V(x)=\gamma x$ is a linear function and $L$ is the symmetric local time of $Z$, then the equation has a unique solution iff $|\gamma|\leq 1$, see  \cite{Harrison+Shepp}, and no solution  if $|\gamma|>1$.
If $\alpha\in(1,2)$, then parameter  $\beta$  can only take values in the interval $(0,\alpha-1)$, see  \cite{IPMS_book, IksanovPilipenko2021skewLevy, Dong+Iksanov+Pilipenko:2024+} and Theorem \ref{thm:lim_skew_Levy_walk}  below.
\footnote{We didn't consider the critical case   $\beta=\alpha-1,$ but we have strong arguments that this case is also impossible.} 
   This statement has a deep connection with It\^o's synthesis theory and a description of all jump entrance laws  for stable processes.

%Below we discuss equation \eqref{eq:guessLimit} for models from Examples \ref{ex:1_1} and \ref{ex:1_2}
\begin{example}
    Consider perturbed RW from Example  \ref{ex:1_1}. 
It is well known that the pair
$\left(\frac{S_\xi(nt)}{\sqrt{n}}, \frac{\sum_{k=0}^{[nt]}1_{S_\xi(k)=0}}{\sqrt{n}}\right)_{t\geq 0}$ converges in distribution to $(B(t), L^B_0(t))_{t\geq 0}$, where $B$ is a standard Brownian motion and $L^B _0$ is its symmetric local time at 0 defined by $$L^B_0(t)\coloneqq \lim_{\varepsilon\to0}\frac{1}{2\varepsilon}\int_0^t{\mathds{1}}_{|B(s)|<\varepsilon}\mathrm{d} s. \quad t\geq 0.$$ Hence, the pair $\left(\frac{X(nt)}{\sqrt{n}}, \frac{\sum_{k=0}^{[nt]}{\mathds{1}}_{X(k)=0}}{\sqrt{n}}\right)_{t\geq 0}=\left(\frac{|S_\xi(nt)|}{\sqrt{n}}, \frac{\sum_{k=0}^{[nt]}{\mathds{1}}_{S_\xi(k)=0}}{\sqrt{n}}\right)_{t\geq 0}$ converges in distribution to  $(|B(t)|, L^B_0(t))_{t\geq 0}= (|B(t)|, L^{|B|}_0(t))_{t\geq 0}$.

Consider now representation \eqref{eq:20_representation_PRW} of $Y.$ We have 
$$\eta_{n,0}=1,\qquad S_{\eta_0}(n)=n,\qquad T^{(0)}(n)=\sum_{i=0}^{n-1}{\mathds{1}}_{Y(i)=0}=S_\eta(T^{(0)}(n)).$$ 
The  scaling for $S_\eta$ is $b_n=n$; $V_0(t)=\lim_{n\to\infty}\frac{S_\eta(nt)}{n}=t.$

 Since perturbed random walks  $X$ and $Y$ have the same distribution, we have convergence
 \[
\begin{split}
\left(\frac{Y(nt)}{\sqrt{n}}, \frac{\sum_{k=0}^{[nt]}{\mathds{1}}_{Y(k)=0}}{\sqrt{n}}\right)_{t\geq 0} 
&\overset{d}{=} \left(\frac{X(nt)}{\sqrt{n}}, \frac{\sum_{k=0}^{[nt]}{\mathds{1}}_{X(k)=0}}{\sqrt{n}}\right)_{t\geq 0} \\
&\Rightarrow (|B(t)|, L^{|B|}_0(t))_{t\geq 0} =: (Z(t) , L^{Z}_0(t))_{t\geq 0}.
\end{split}
\]
 % \[
 % \left(\frac{Y(nt)}{\sqrt{n}}, \frac{\sum_{k=0}^{[nt]}{\mathds{1}}_{Y(k)=0}}{\sqrt{n}}\right)_{t\geq 0}\overset{d}{=}\left(\frac{X(nt)}{\sqrt{n}}, \frac{\sum_{k=0}^{[nt]}{\mathds{1}}_{X(k)=0}}{\sqrt{n}}\right)_{t\geq 0}\Rightarrow (|B(t)|, L^{|B|}_0(t))_{t\geq 0}=: (Z(t) , L^{Z}_0(t))_{t\geq 0} \]

Note that $S_\xi(T^{\text{normal}}(n))=Y(n)-\sum_{k=0}^{n}{\mathds{1}}_{Y(k)=0}$.
So we have convergence of triples
\[
\left(\frac{Y(nt)}{\sqrt{n}}, \frac{\sum_{k=0}^{[nt]}{\mathds{1}}_{Y(k)=0}, }{\sqrt{n}}, \frac{S_\xi(T^{\text{normal}}(nt))}{\sqrt{n}}\right)_{t\geq 0} \Rightarrow   (Z(t) , L^{Z}_0(t), Z(t)- L^{Z}_0(t))_{t\geq 0}.
\]
On the other hand we know that $(\frac{S_\xi(T^{\text{normal}}(nt))}{\sqrt{n}})_{t\geq 0}\Rightarrow (W(t))_{t\geq 0}$, where $W$ is a standard Brownian motion due to Donsker's theorem and   \eqref{eq:conv-stab-proc}. Therefore,  representation \eqref{eq:guessLimit} for reflected Brownian motion takes the well-known form, e.g. \cite{PilipenkoRSDE}:
\begin{equation}
    \label{eq:refl_Z}
    Z(t)=W(t)+L^Z_0(t), t\geq 0,
\end{equation}
where $L^Z_0$ is a symmetric local time of $Z$ at 0. 
\end{example}
\begin{example}\label{expl:skew2}
Consider the model from  Example  \ref{ex:1_2}. % and forget for a while that we have already known that the scaling limit exists.  
It can be seen that $(|Y(n)|)$ has the same distribution as a random walk from Example \ref{ex:1_1}. This implies that  the sequence $\Big\{\big(\frac{Y(nt)}{\sqrt{n}}, \frac{\sum_{k=0}^{[nt]}{\mathds{1}}_{Y(k)=0}, }{\sqrt{n}}\big)_{t\geq 0}, n\geq 1\Big\}$ is  weakly relatively compact;    any of its  limit point $(Z,L)$ is a continuous process; the process   $L$ is the symmetric local time of $Z.$

Recall that   
${\mathop{\rm P}}(\eta_{n,0}= 1)=p, {\mathop{\rm P}}(\eta_{n,0}= -1)=q$. So, $\lim_{n\to\infty}\frac{S_\eta(nt)}{n}=(p-q)t=:V_0(t) $ a.s. Hence, any limit process $Z$ satisfies the equation
\begin{equation}
    \label{eq:skew_SDE_Z}
    Z(t)=W(t)+\gamma L^Z_0(t), t\geq 0,
\end{equation}
where $\gamma=(p-q)\in[-1,1],$ $W$ is a standard Brownian motion, $L^Z_0$ is a symmetric local time of $Z$ at 0. 

We would like to   exclaim that $Z$ is a  solution to the stochastic differential equation  \eqref{eq:skew_SDE_Z} with a local time that appeared in the paper by Harrison and Shepp  \cite{Harrison+Shepp}, and that this equation has a unique strong solution. At this point, it is necessary to be cautious. Let $({\mathcal F}_t)=({\mathcal F}^Z_t)$ be a filtration generated by $Z $ augmented by sets of zero-probability.   We know that $W$ is a Brownian motion. Is $W$ an $({\mathcal F}_t)$-Brownain motion? Since $L^Z_0$ is a local time of $Z,$ then $W(t)=Z(t)-L^Z_0(t)$ is ${\mathcal F}_t$-measurable random variable, that is, $W$ is $({\mathcal F}_t)$-adapted process. To treat  \eqref{eq:skew_SDE_Z} as an SDE  we need to ensure that increments $W(t+s)-W(t), s\geq 0,$ are independent of ${\mathcal F}_t$ for every $t\geq 0.$ This can be shown easily for considered model. However, we do not know   arguments showing that there is a unique solution to \eqref{eq:skew_SDE_Z} for some Brownian motion $W$ if we do not assume that $W$ is an $({\mathcal F}_t)$-Brownian motion.
 \end{example}

%%%%%%%%%%%%%%%%%%%%%%%%%%%%%%%%%%%%%%%%%%%%%%%%%%%%%%%%%%%%%%%%%%%%%%%%%%%%%%%%%%%%%%%

%%%%%%%%%%%%%%%%%%%%%%%%%%%%%%%%%%%%

\section{Reflecting screens.} \label{sec:reflection}
In this section, we assume that if the perturbed random walk enters  the negative half-line, perturbations act to pull the walk back into the positive half-line. These results will be useful for studying perturbed random walks on ${\mathbb Z}$ and on graphs, see \S \ref{sec:spider}.
We also assume that ${\mathop{\rm E}}\xi=0$ and ${\mathop{\rm Var}} \xi\in(0,\infty)$, making it reasonable to expect the limiting process will be non-negative and behave like Brownian motion when positive. As suggested by our earlier guess for the limit process \eqref{eq:guessLimit}, an additional non-decreasing term is likely needed to push the limit process upward  whenever it touches zero. This idea aligns with Skorokhod's reflection problem and its extensions. For further details and definitions, see \cite[\S 1.1.1]{IPMS_book}.

\begin{definition}
\label{defn:SKORreflProblem}
Let  $f\in {\mathcal D} , f(0)\geq0$. A pair  $g,l\in{\mathcal D}$ is called to be a solution of the Skorokhod problem    for $f$ if
\[
g(t)=f(t)+l(t), \ t\geq0,
\]
where $g $ is a non-negative function, $l$ is non-decreasing, $l(0)=0$, and $l$ may increase only at the moments when $g$ visits 0, i.e.,
  \begin{equation}
\label{eq:S4}
  \int\limits^\infty_0{\mathds{1}}_{g(s)>0}dl(s)=0.
\end{equation}
%\end{enumerate}
\end{definition}
It is well known that for any $f\in{\mathcal D}$ with $f(0)\geq 0,$ there is a unique solution to the Skorokhod problem. The solution is given by the formula
\begin{equation}
     \begin{aligned}  \label{eq:Skor_sol}
g(t)= \Gamma(f)(t)\coloneqq f(t) - \min_{s\in[0,t] }(f(s))\wedge 0,\\
l(t)=- \min_{s\in[0,t] }(f(s))\wedge 0.
\end{aligned} 
\end{equation}
We will call the map $\Gamma:{\mathcal D}\to{\mathcal D}$ the Skorokhod map. Important property of    the Skorokhod map is its continuity.

\begin{example}[Lindley's model and heavy traffic limits]
Consider a server tasked with processing information bits. At each discrete time tick, data arrives for processing, and  the server processes a portion of this data (first receiving, then processing). Let $\xi^+_k$ denote the quantity of arrivals at the $k$th tick, and $\xi^-_k$ represent the quantity of bits that can be processed at that time. Define $\xi_k \coloneqq  \xi^+_k - \xi^-_k$, and assume that $(\xi_k)$ is a sequence of independent and identically distributed random variables with ${\mathop{\rm E}}\xi = 0$ and ${\mathop{\rm Var}} \, \xi  = \sigma^2 \in (0, \infty)$. The queue length, or the number of bits awaiting processing at the $k$th tick, evolves according to Lindley’s recursion:
\[
\widetilde Y(k+1)=(\widetilde Y(k)+\xi_k)\vee 0, k\geq 0.
\]
It can be seen that   the process $\widetilde Y(t), t\geq 0,$ is a solution of the Skorokhod problem for $S_\xi$ (we recall that we set $\widetilde Y(t)\coloneqq \widetilde Y(n), t\in[n,n+1)$).    
Therefore, $\frac{\widetilde Y(nt)}{\sqrt{n}}=\Gamma (\frac{S_\xi(n\cdot)}{\sqrt{n}})(t), t\geq 0$. Since Skorokhod map is continuous in ${\mathcal D}$, convergence
\[\Big(\frac{\widetilde Y(nt)}{\sqrt{n}}\Big)_{t\geq0}\Rightarrow \Big(\sigma B(t)\Big)_{t\geq0}, n\to\infty,\]
in ${\mathcal D}$ is a simple corollary of Donsker's theorem and continuous mapping theorem. 

Note, that $\widetilde Y$ is not a perturbed RW with membrane $A={\mathbb Z}\setminus {\mathbb N}_0$ in the sense of \eqref{eq:X_repr} or \eqref{eq:20_representation_PRW}. Indeed, if $\widetilde Y(n)+\xi_{n+1}<0,$ i.e., if $\widetilde Y$ ``wishes'' to jump into negative half-line, then we ``skip'' this step and jump  to 0.
This scenario was discussed in Lemma \ref{lem:equivalent perturbed with skips}. The corresponding perturbed RW, where $Y$ has transition probabilities 
$ p_{x,y}=\begin{cases}
    {\mathop{\rm P}}(\xi=y-x), & x\geq 0;\\
    {\mathds{1}}_{y=0}, & x<0.
    \end{cases}$

\end{example}
\begin{example}[Oscillating random walk]\label{expl:Oscillating} Consider random walk $X$ with transition probabilities
\begin{equation}
    {\mathop{\rm P}}(X(n+1)=x+y | X(n)=x)=\begin{cases}
    {\mathop{\rm P}}(\xi=y), & x>0;\\
    {\mathop{\rm P}}(\eta=y), &  x\leq 0;
\end{cases}
\label{eq:oscillating}
\end{equation}
 where ${\mathop{\rm E}}\xi=0, \quad {\mathop{\rm Var}}\, \xi=\sigma^2\in(0,\infty),\quad {\mathop{\rm E}}\eta\in(0,\infty).$

Let us  apply Skorokhod's representation theorem, Donsker's theorem, and construct a probability space, a sequence of 
copies $S_\xi^n\overset{d}{=}{S_\xi},  \quad n\in {\mathbb N}$, and  a Brownian motion $B,$  defined on this probability space, such that

\begin{equation}
    \label{eq:conv_Skor_repr1}\Big(\frac{S_\xi^n(nt)}{\sigma\sqrt{n}}\Big)_{t\geq 0}\to \Big(B(t)\Big)_{t\geq 0}  \quad n\to\infty,
\end{equation} 
almost surely in ${\mathcal D}.$ Since the limit is continuous, convergence in ${\mathcal D}$ is equivalent to the locally uniform convergence.

Without loss of generality we may assume that the original random variables $(\eta_k)_{k\geq 0}$ are also defined at this probability space. So 
\begin{equation}
    \label{eq:conv_Skor_repr2} \Big(\frac{S_\eta (nt)}{n}\Big)_{t\geq 0}\to \Big(({\mathop{\rm E}}\eta) t\Big)_{t\geq 0}, \quad n\to\infty, \quad \text{a.s.}
\end{equation} 
locally uniformly on compact sets.

To analyze convergence of the scaling limits, let us use representation \eqref{eq:20_representation_PRW} for copies $Y_n$ of $X$ that are functions  of  $S_\xi^n$ and $S_\eta $:
\[
\frac{Y_n(nt)}{\sigma\sqrt{n}}= \frac{S_\xi^n(T^\text{normal}_n(nt))}{\sigma\sqrt{n}}  + \frac{S_\eta (T^{\text{critical}}_n(nt))}{\sigma\sqrt{n}} = \frac{S_\xi^n(T^\text{normal}_n(nt))}{\sigma\sqrt{n}}  + \frac{S_\eta (\sigma\sqrt{n}\frac{T^{\text{critical}}_n(nt) }{\sigma\sqrt{n}})}{\sigma\sqrt{n}} ,
\]
where $T^\text{normal}_n(k)=\sum_{i=0}^{k-1}{\mathds{1}}_{Y_n(k)>0},\quad T^{\text{critical}}_n(k)=\sum_{i=0}^{k-1}{\mathds{1}}_{Y_n(i)\leq 0}.$

We are going to prove convergence for every $\omega$ from the set where we have convergence \eqref{eq:conv_Skor_repr1} and \eqref{eq:conv_Skor_repr2}. The result will not use any specific property of a Brownian motion except of continuity. 

Let us denote $B_n(t)\coloneqq \frac{S_\xi^n(nt)}{\sigma\sqrt{n}}$,  $H_n(t)\coloneqq \frac{S_\eta(\sigma \sqrt{n}t)}{\sigma \sqrt{n}}$.

Hence
\begin{equation}
    \label{eq:oscillating repr}
    \frac{Y_n(nt)}{\sigma\sqrt{n}}=B_n(\frac{T_n^\text{normal}(nt)}{n}) + H_n(\frac{T^{\text{critical}}_n(nt)} {\sigma\sqrt{n} }).
\end{equation}

The following estimate is important part of the proof.
For   any $N$ and $\omega$:
\begin{equation}\begin{aligned}
    \label{eq:estimate_S}
    & S_\eta(T^{\text{critical}}(N)-1)   \leq -\min_{0\leq i  \leq T^\text{normal}(N)}(S_\xi(i)\wedge 0) \\
     &\leq  S_\eta(T^{\text{critical}}(N))+ \max_{0\leq i\leq T^\text{normal}(N)} (\xi_{i+1}-\xi_{i })\wedge0+ \max_{0\leq i\leq j\leq T^{\text{critical}}(N)}(S_\eta(i)-S_\eta(j)).   
     \end{aligned}\end{equation}
 
For the detailed proof of these inequalities see   \cite{Pilipenko+Sarantsev:2024, IPMS_book}. The proof is inductive and relies on a careful examination of the construction of $Y$. The result has a deterministic nature and holds for any  sequences $(\xi_k), (\eta_k)$.

Hence 
\begin{equation}
\label{eq:ineqS}
\begin{aligned}  
H_n\left(\frac{T^{\text{critical}}_n(nt)-1}{\sigma\sqrt{n}}\right) 
&\leq \min_{s\in[0,t]}\Big(B_n\Big(\frac{T_n^\text{normal}(ns)}{n}\Big)\wedge 0\Big) \\
&\leq H_n\left(\frac{T^{\text{critical}}_n(nt)}{\sigma\sqrt{n}}\right) + \max_{s\in[0,t]}|B_n(s)-B_n(s-)| \\
&\quad + \max_{0\leq t_1\leq t_2\leq t} \left(H_n\left(\frac{T^{\text{critical}}_n(nt_1)}{\sigma\sqrt{n}}\right)-H_n\left(\frac{T^{\text{critical}}_n(nt_2)}{\sigma\sqrt{n}}\right)\right).
\end{aligned}
\end{equation}
% \begin{equation}\begin{aligned}  
% & H_n(\frac{T^{\text{critical}}_n(nt)-1} {\sigma\sqrt{n} })   \leq  \min_{s\in[0,t]}\Big(B_n(\frac{T_n^\text{normal}(ns)}{n})\wedge 0\Big)\leq \\
% &H_n(\frac{T^{\text{critical}}_n(nt)} {\sigma\sqrt{n} })+ \max_{s\in[0,t]}|B_n(s)-B_n(s-)|+\max_{0\leq t_1\leq t_2\leq t} \Big(H_n(\frac{T^{\text{critical}}_n(nt_1) } {\sigma\sqrt{n} })-H_n(\frac{T^{\text{critical}}_n(nt_2) } {\sigma\sqrt{n} })\Big)\label{eq:ineqS}
% \end{aligned}\end{equation}

 Since 
 \[
\begin{split}
\limsup_{n\to\infty}\Big(-\min_{s\in[0,t]}\Big(B_n\Big(\frac{T_n^\text{normal}(ns)}{n}\Big)\wedge 0 \Big)\Big)
&\leq \lim_{n\to\infty}\Big(-\min_{s\in[0,t]}B_n(s)\wedge 0\Big) \\
&=-\min_{s\in[0,t]}(B (s)\wedge 0)=-\min_{s\in[0,t]} B (s)<\infty.
\end{split}
\]
% \[
% \limsup_{n\to\infty}\Big(-\min_{s\in[0,t]}(B_n(\frac{T_n^\text{normal}(ns)}{n})\wedge 0 \Big)\leq
%  \lim_{n\to\infty}\Big(-\min_{s\in[0,t]}B_n(s)\wedge 0\Big)=-\min_{s\in[0,t]}(B (s)\wedge 0)=-\min_{s\in[0,t]} B (s)<\infty,
% \]
the first inequality implies that 
$\limsup_{n\to\infty}H_n(\frac{T^{\text{critical}}_n(nt)-1} {\sigma\sqrt{n} })<\infty.$ This and \eqref{eq:conv_Skor_repr2} implies
\begin{equation}\begin{aligned}\label{eq:991}
\limsup_{n\to\infty} \frac{T^{\text{critical}}_n(nt)} {\sigma\sqrt{n} }<\infty,
\end{aligned}\end{equation}
$\lim_{n\to\infty}\frac{T^{\text{critical}}_n(nt)}{n}=0$ and $\lim_{n\to\infty}\frac{T^\text{normal}_n(nt)}{n}=t$  for any $t\geq 0$. Since all functions are non-decreasing and the limits are continuous,  convergence is  locally uniform in $t$. At first, this implies locally uniform convergence
\[
\lim_{n\to\infty} B_n(\frac{T_n^\text{normal}(nt)}{n}) = B(t).
\]
At second, \eqref{eq:991} and \eqref{eq:conv_Skor_repr2} imply that
\[
\lim_{n\to\infty}\max_{0\leq t_1\leq t_2\leq t} \Big(H_n(\frac{T^{\text{critical}}_n(nt_1) }{\sigma\sqrt{n}})-H_n(\frac{T^{\text{critical}}_n(nt_2) }{\sigma\sqrt{n}})\Big)=0.
\]
Hence 
\[
\lim_{n\to\infty}   H_n(\frac{T^{\text{critical}}_n(nt)}{\sigma\sqrt{n}})= - \min_{s\in[0,t]} B(s) 
\]
uniformly in compact sets.

Thus, we obtain locally uniform convergence of $\frac{Y_n(nt)}{\sigma\sqrt{n}}$ to $B(t)- \min_{s\in[0,t]}B(s)$ as $n\to\infty.$ Moreover, we obtain a limit theorem for the number of visits to ${\mathbb Z}\setminus {\mathbb N}$:
\[
\lim_{n\to\infty}\frac{T^{\text{critical}}_n(nt)}{\sigma\sqrt{n}} = ({\mathop{\rm E}}\eta)^{-1}  \min_{s\in[0,t]}B(s).
\]
Returning  to the original sequence $X$ we obtain convergence of pairs
\[
\Big(\frac{X(nt)}{\sigma \sqrt{n}}, \frac{T^{\text{critical}}(nt)}{\sigma \sqrt{n}}\Big)_{t\geq0}\Rightarrow \Big( \Gamma (B)(t),  ({\mathop{\rm E}}\eta)^{-1}  \min_{s\in[0,t]}B(s)\Big)_{t\geq0}, n\to\infty.
\]

 \end{example}

   We stress that  the result from  the previous example was proved as an application of Skorokhod's representation theorem, Donsker's theorem, the law of large numbers, and real analysis arguments. 
    Consideration  of more  involved models are similar, see \cite[\S3.2.1]{IPMS_book}. Informally, the reflected Brownian motion appears in models, where the $n$th jump from the membrane ${\mathbb Z}\setminus {\mathbb N}$ is  negligible relative to the accumulated sum of jumps from  the membrane as $n\to\infty.$ 
If this assumption fails, say if jumps from the membrane are heavy-tailed, then  other scaling limits may appear, see next example.

\begin{example}\label{expl:refl_Levy}
    Consider an oscillating random walk given by \eqref{eq:oscillating}, where $\xi $ is the same as in the previous example and $\eta$ belongs to the domain of attraction of a positive $\alpha$-stable law with $\alpha\in(0,1).$ Similarly to the previous example we apply Skorokhod's representation theorem and construct copies of sequences such that
    $B_n(t)\coloneqq \frac{S_\xi^n(nt)}{\sigma\sqrt{n}}$,  $H_n(t)\coloneqq \frac{S_\eta^n(b_n t)}{\sigma \sqrt{n}}, t\geq 0$  converge almost surely in  ${\mathcal D}$ to an independent Brownian motion $B(t), t\geq 0$ and an $\alpha$-stable subordinator $H(t), t\geq 0$, respectively. Here $(b_n)$ is a well-chosen sequence of positive numbers. It is well known that $(b_n)$ is a regularly varying sequence with parameter $\alpha/2$. 

Below we will consider only $\omega$ from the corresponding set of probability 1 from the Skorokhod theorem.

We still have representation \eqref{eq:oscillating repr}, and inequalities \eqref{eq:estimate_S}, \eqref{eq:ineqS}. Since $B$ is continuous  and $H$ is increasing we get
\begin{equation}
\label{eq:1031}
\limsup_{n\to\infty} H_n(\frac{T^{\text{critical}}_n(nt)-1}{b_n})\leq M(t)\coloneqq  -\min_{s\in[0,t]}B(s)\leq  \liminf_{n\to\infty}
H_n(\frac{T^{\text{critical}}_n(nt)}{b_n}). 
\end{equation}
According to \cite[Proposition 6.5, Chapter 3]{Ethier+Kurtz:1986}, if a sequence of (non-random) c\`adl\`ag  functions $(f_n)$ converges to $f$ in ${\mathcal D}$, then for any $t> 0$ and sequence $(t_n)$ with $\lim_{n\to\infty} t_n=t$,  the limit points  of the sequence $(f_n(t_n))$ can be either $ f(t-)$ or $f(t)$. In particular, if $t$ is a point of continuity of $f,$ then $\lim_{n\to\infty} f_n(t_n)=f(t).$
Using this result with   inequality \eqref{eq:1031}, it follows that 
\begin{equation}
     \label{eq:1037}
     \limsup_{n\to\infty}  \frac{T^{\text{critical}}_n(nt)}{b_n}<\infty.
\end{equation}
Moreover, for any fixed $t$, any limit point $L(t)$ of the sequence $( \frac{T^{\text{critical}}_n(nt)}{b_n})_{n\geq 1}$ must satisfy the inequality
\begin{equation}
    \label{eq:ineq_M_H}
    H(L(t)-)\leq M(t)\leq  H(L(t)) .
\end{equation}
Since $H$ is strictly increasing and unbounded, there is a unique $L(t)$ that satisfies this property, 
$L(t)=H^{-1}(M(t))$.  Due to uniqueness we have   convergence $ \lim_{n\to\infty}  \frac{T^{\text{critical}}_n(nt)}{b_n}=L(t).$  

The process $H^{-1}$ is continuous    because $H$ is strictly increasing. So, $L$ is continuous as well and we have even the locally uniform convergence $ \lim_{n\to\infty}  \frac{T^{\text{critical}}_n(nt)}{b_n}=H^{-1}(M(t))$ since all processes are non-decreasing and the limit is continuous.  

Recall that $(b_n)$ is regularly varying sequence of index $\alpha/2,$ so $b_n=o(n), n\to \infty.$ The bound \eqref{eq:1037} implies locally uniform convergence 
\begin{equation*}
    \lim_{n\to\infty}\frac{T^{\text{critical}}_n(nt)}{n}=0, \quad \lim_{n\to\infty}\frac{T^\text{normal}_n(nt)}{n}=t, 
\quad \lim_{n\to\infty} B_n(\frac{T_n^\text{normal}(nt)}{n})=B(t).
\end{equation*}

Although $\lim_{n\to\infty} H_n=H$ and $ \lim_{n\to\infty} \frac{T^{\text{critical}}_n(n\cdot )}{b_n}=H^{-1}\circ M $ in ${\mathcal D}$, proving the convergence of compositions
\begin{equation}
	\label{eq:conv_H_composition_critical}
	  \lim_{n\to\infty} H_n\Big( \frac{T^{\text{critical}}_n(n\cdot )}{b_n}\Big)=H\circ H^{-1}\circ  M   \text{ in }  {\mathcal D},
\end{equation}
requires additional reasoning. This is because composition is generally not continuous in the space of c\`adl\`ag functions. For detailed proofs, see \cite{Pilipenko+Sarantsev:2024} or \cite[\S 3.2.2, \S 3.2.3]{IPMS_book}.
The only   property of trajectories that  was used in the corresponding proof is the absence of common points of discontinuity of $H$ and $M^{-1}$, which is true for a.a. $\omega$ because processes $H$ and $M^{-1}$ are independent subordinators. 

It is well known that convergence $\lim_{n\to\infty} f_n=f$ and $\lim_{n\to\infty} g_n=g$ in ${\mathcal D}$ implies convergence $\lim_{n\to\infty} (f_n+g_n) =f+g$ in ${\mathcal D}$
if $f$ and $g$ do not have same points of discontinuity, e.g. \cite[ Theorem 4.1]{Whitt:1980}. This proves a.s. convergence in ${\mathcal D}$ of $\Big(\frac{Y_n(nt)}{\sigma\sqrt{n}}\Big)_{t\geq 0}=\Big( B_n (\frac{T_n^\text{normal}(nt)}{n} ) +  H_n ( \frac{T^{\text{critical}}_n(nt )}{b_n} )\Big)_{t\geq 0}$ and consequently implies convergence in distribution:
\begin{equation}
    \label{ex:fla_gen_Skor}
    \left(\frac{Y (n t)}{\sigma\sqrt{n}}\right)_{t\geq 0}\Rightarrow \left(B(t)+ H\circ H^{-1} \circ M(t)\right)_{t\geq 0}, \quad n\to\infty.
\end{equation}
    
\end{example}

\textbf{Discussion and possible generalizations.}

The concept of dividing dynamics into ‘normal’ and ‘critical’ (or ‘compensating’) modes, and employing representations such as \eqref{eq:oscillating repr}, was introduced in \cite{Pilipenko+Sarantsev:2024}, where a deterministic limit theorem for switching dynamics was established. Once Skorokhod’s representation theorem is applied, all arguments become primarily deterministic, and functional limit theorems follow from the continuity of maps in the Skorokhod space ${\mathcal D}$ of càdlàg functions. 

This methodology is effective for triangular arrays. For instance, consider a sequence of random walks $(X_n(k))_{k\geq 0}$ with transition probabilities:
\[
 p^n_{x,y}=\begin{cases}
    {\mathop{\rm P}}(\xi_n=y-x), & x\geq 0;\\
    {\mathop{\rm P}}(\eta_n =y-x), & x<0, 
    \end{cases}
\]
where $\xi_n$ and $\eta_n$ are integer-valued random variables. Assume there exist scaling sequences $(a_n)$ and $(b_n)$ such that the partial sums $\frac{S_{\xi_n}(n\cdot)}{a_n}$ and $\frac{S_{\eta_n}(n\cdot)}{b_n}$ converge in distribution as $n\to\infty$ to a continuous Lévy process $B$ and a strictly increasing subordinator $H$, respectively. 
There are two key points to address. First, the time spent in the critical mode must satisfy $\lim_{n\to\infty}\frac{T^{\text{critical}}_n(nt)}{n}=0$. While this holds readily in the stationary case where $\xi_n\overset{d}{=} \xi$ and $\eta_n\overset{d}{=} \eta$, it is not guaranteed for general sequences $\xi_n, \eta_n$. In certain degenerate cases, sticky-reflected Brownian motions may emerge as the limiting process; see \cite{Pilipenko+Sarantsev:2024, IPMS_book, Pilipenko+Prikhodko:2020}.

The second point concerns the convergence of compositions in \eqref{eq:conv_H_composition_critical}. This convergence is straightforward when $H$ is linear, but requires a more detailed analysis if $H$ is a jump process, particularly regarding the matching of discontinuity points; see \cite{Pilipenko+Sarantsev:2024, IPMS_book}.

\section{Scaling limits of random walks on a spider}\label{sec:spider}
A random walk on ${\mathbb R}$ can be viewed as a random walk on a graph consisting of two rays with a common vertex. As we will see in this section, the transition from a graph with two rays to a star graph with $d$ rays does not add any complexity; indeed, the general case is often more convenient in terms of notation. The constructions of the previous section regarding reflection (the one-dimensional case) serve as the ``building blocks'' for a general graph.

Let $\xi_i$, $1\leq i\leq d,$ be zero-mean integer-valued random variables with finite variances $v_i^2={\mathop{\rm Var}}\, \xi_i \in(0,\infty).$ Consider a Markov chain $(R(n), l(n))_{n\geq 0}$ on the state space $E\coloneqq {\mathbb Z} \times \{1,\dots,d\}$ satisfying the following properties:
\begin{itemize}
    \item ${\mathop{\rm P}}((R(n+1), l(n+1))=(x+y,i) \mid (R(n), l(n))=(x,i)) = {\mathop{\rm P}}(\xi_i=y)$ if $x\in{\mathbb N}$;
    \item ${\mathop{\rm P}}(R(n+1)\in{\mathbb N} \mid R(n)\leq 0) = 1$;
    \item The states in ${\mathbb N}\times\{1,\dots,d\}$ communicate.
\end{itemize}

The dynamics of $(R,l)$ can be interpreted as a random walk of a particle on the union of $d$ horizontal lines. The second coordinate $l(n)$ is the number of the line (label) and the first coordinate $R(n)$ (radius) is the position of the particle on the line at time $n$. If the radius is positive and the particle is on the $i$-th line, it jumps along the horizontal direction independently of the past with a distribution $\xi_i$. If the first component reaches a non-positive value, then at the next step, the particle may change lines and the radius becomes positive.

Associated with the Markov chain $(R,l)$ is a random walk on the coordinate axes of a $d$-dimensional space defined by:
\begin{equation}
    \label{eq:R-to_x}
    X(n)\coloneqq (X_1(n),\dots,X_d(n))\coloneqq (R(n){\mathds{1}}_{l(n)=1},\dots, R(n){\mathds{1}}_{l(n)=d}), \quad n\geq 0.
\end{equation}
By ${\mathcal Z}^d$ and ${\mathcal N}^d$, we denote the intersection of the coordinate axes with ${\mathbb Z}^d$ and ${\mathbb N}^d$, respectively. The goal of this section is to investigate the scaling limits of $X$. At the end, we obtain scaling limits for a perturbed random walk on ${\mathbb Z}$ as a simple corollary of the general result.

Note that only one coordinate of $X$ may be non-zero. If $X_i(n)\in{\mathbb N}$, then the next jump is in the direction of the $i$-th axis and the distribution of the jump is $\xi_i$. If $X(n)\notin{\mathcal N}^d$, then $X$ enters ${\mathcal N}^d$ at the next step, $X(n+1)\in{\mathcal N}^d$. 

The sequence $X$ is not, in general, a Markov chain on ${\mathcal Z}^d$; the Markov property may fail at $X(n)=0$ since $X(n+1)$ may depend on the line label $l(n)$. On the other hand, each Markov chain on ${\mathcal Z}^d$ that has i.i.d. jumps in ${\mathcal N}^d$ and enters ${\mathcal N}^d$ immediately after exiting can be represented via formula \eqref{eq:R-to_x} with some $(R,l)$.

To construct a representation of $X$ similar to \eqref{eq:20_representation_PRW}, we introduce the exit (from ${\mathcal N}^d$) and entrance (to ${\mathcal N}^d$) stopping times:
\begin{equation*}\begin{aligned}
\sigma_0 &= \inf\{k \geq 0 : R(k) \leq 0\}, \\
\sigma_{n+1} &= \inf\{k > \sigma_n : R(k) \leq 0 \}, \quad n \geq 0, \\
\tau_n &= \inf\{k > \sigma_n : R(k) \in {\mathbb N} \} = \sigma_n + 1, \quad n \geq 0.
\end{aligned}\end{equation*}
The moment $\sigma_n$ is the instant when $R$ visits the non-positive part for the $n$-th time. We define the following counting processes:
\begin{equation*}\begin{aligned}
T^{\text{normal},i}(n) &\coloneqq \sum_{k=0}^{n-1}{\mathds{1}}_{X_i(k)\in{\mathbb N}}, \quad 
T^{\text{exit},i}(n) \coloneqq \sum_{k=0}^{n-1}{\mathds{1}}_{R(k)\leq 0,\, l(k)=i}, \\
T^{\text{critical}}(n) &\coloneqq \sum_{k=0}^{n-1}{\mathds{1}}_{R(k)\leq 0} = \sum_{i=1}^d T^{\text{exit},i}(n).
\end{aligned}\end{equation*}
Set
\begin{equation*}\begin{aligned}
\xi_{n,i} &\coloneqq X_i((T^{\text{normal},i})^{-1}(n-1)) - X_i((T^{\text{normal},i})^{-1}(n-1)-1), \quad n \geq 1, \\
S_{\xi_i}(n) &\coloneqq X_i(0) + \sum_{k=1}^n \xi_{k,i}.
\end{aligned}\end{equation*}
The variable $\xi_{n,i}$ is the size of the $n$-th jump from the positive part of the $i$-th ray. The strong Markov property implies that $(\xi_{n,i})_{n\geq 1}, 1 \leq i \leq d$ are independent sequences of i.i.d. random variables, $\xi_{n,i} \overset{d}{=} \xi_i$.

%%%%%%%%%%%%%%%%%%%%%%%%%%%%%%%%%%%%%%%%%%%%%%%%%%%%%%%%%%%%

%%%%%%%%%%%%%%%%%%%%%%%%%%%%%%%%%%%%%%%%%%%%%%%%%%%%%%%

{ Denote
\begin{equation*}\begin{aligned}
H_i(N)\coloneqq \sum_{k=0}^{N-1}\Big(X_i(\sigma_k+1)-X_i(\sigma_k)\Big)= \sum_{k=0}^{N-1}\Big(X_i(\tau_k)-X_i(\sigma_k)\Big).
\end{aligned}\end{equation*}
Then 
\begin{equation}\begin{aligned}
    \label{eq:WalshRW_repr}
X_i(n)=S_{\xi_i}(T^{\text{normal},i}(n))+ H_i(T^{\text{critical}}(n)).
\end{aligned}\end{equation}
}

Notice that $T^{\text{exit},i}(n)$ does not exceed the number of descending ladder epochs of the unperturbed random walk $S_{\xi_i}$ within the interval $[0, n]$. Similarly to the proof of Lemma \ref{lem:neglidg_critical}, we obtain the a.s. convergence:
\begin{equation}
    \label{eq:crit-to0}
\lim_{n\to\infty}\sum_{i=1}^d\frac{T^{\text{exit},i}(n)}{n}=0.
\end{equation} 

\begin{lemma}
    \label{lem: compactness of T}
    The sequence of stochastic processes $\left\{\left(\frac{T^{\text{normal},i}(nt)}{n}\right)_{1\leq i\leq d, \ t\geq 0}, n\geq 1\right\}$ is weakly relatively compact in ${\mathcal D}$, and any limit point $( T^{\text{normal},i}_\infty(t))_{1\leq i\leq d, \ t\geq 0}$ is a continuous process such that $\sum_{i=1}^d T^{\text{normal},i}_\infty(t)=t$ for all $t\geq 0$.
\end{lemma}

\begin{proof}
    Weak relative compactness and continuity of the limit process follow from the observation that for every $s=\frac{k}{n}$ and $t=\frac{l}{n}$, we have control of the modulus of continuity:
\[
\left| \frac{T^{\text{normal},i}(nt)}{n}-\frac{T^{\text{normal},i}(ns)}{n}\right|\leq |t-s|.
\]
 The proof of $\sum_{i=1}^d T^{\text{normal},i}_\infty(t)=t$ follows from \eqref{eq:crit-to0} and the fact that 
 $$T^{\text{critical}}(n)+\sum_{i=1}^d T^{\text{normal},i}(n)=n.$$
\end{proof}

{
It can be shown that the following counterpart of \eqref{eq:estimate_S} holds true:
\begin{equation}
\begin{aligned}
    \label{eq:estimate_S_Walsh1}
     H_i(T^{\text{critical}}(N)-1)
  &    \leq -\min_{0\leq k \leq N}\left(S_{\xi_i}(T^{\text{normal},i}(k))\wedge 0\right)  \\& 
     \leq  
     H_i(T^{\text{critical}}(N))+ \max_{0\leq k\leq T^{\text{normal},i}(N)}|\xi_{k+1,i}-\xi_{k,i }|.    
\end{aligned}
\end{equation}
}
We leave the verification of this fact to the reader.

The plan of investigation is as follows. {Assume that a law of large numbers is satisfied for $(H_i(N))_{N\geq 1}, 1\leq i\leq d$. Namely, suppose that there are positive constants $c_i, 1\leq i\leq d,$ such that:
\begin{equation}\begin{aligned}
\label{eq:LLN_Hi}
 \lim_{N\to\infty}\frac{H_i( N )}{ v_iN}=c_i, \quad \text{a.s.}
\end{aligned}\end{equation}
This implies the locally uniform convergence:
\begin{equation}\begin{aligned}
\label{eq:LLN_Hi_uniform}
 \forall T>0, \quad \lim_{N\to\infty}\max_{t\in[0,T]}\left|\frac{H_i( Nt )}{ v_iN}-c_it\right|=0, \quad \text{a.s.}
\end{aligned}\end{equation}
}

Then, applying inequality \eqref{eq:estimate_S_Walsh1}, Donsker's theorem, and Skorokhod's representation theorem, we will show that the difference between the l.h.s. and r.h.s. of \eqref{eq:estimate_S_Walsh1} is $o(\sqrt{N})$ as $N\to\infty$. Finally, we will apply \eqref{eq:LLN_Hi_uniform} to \eqref{eq:WalshRW_repr}, obtain an equation for the scaling of $T^{\text{critical}}$ and $T^{\text{normal},i}$, and utilize Theorem \ref{thm:baryaktar}. 

 \begin{theorem}
     \label{thm:pert_Walsh}
   Assume that \eqref{eq:LLN_Hi} holds true. Then the sequence of stochastic processes $\left(\frac{X_1(nt)}{ v_1\sqrt{n}},\dots,\frac{X_d(nt)}{ v_d\sqrt{n}}\right)_{t\geq 0}$ converges in distribution to the Walsh Brownian motion started from 0 with parameters $p_i=\frac{c_i}{c_1+\dots+c_d}, 1\leq i\leq d$. 
 \end{theorem}

%%%%%%%%%%%%%%%%%%%%%%%%%%%%%%%%%%%%%%%%%%%%%%%%%%%%%%%%%%%%%%%%%%%%%
\begin{proof}
It is sufficient to prove that any subsequence of $\left\{\left(\frac{X_1(n\cdot)}{ v_1\sqrt{n}},\dots,\frac{X_d(n\cdot)}{ v_d\sqrt{n}}\right)\right\}_{n\geq 1}$ contains a sub-subsequence that converges to the Walsh process. By Lemma \ref{lem: compactness of T}, and without loss of generality, we assume that the sequence of stochastic processes $\left(\frac{S_{\xi_i}(nt)}{\sqrt{n}},\frac{T^{\text{normal},i}(nt)}{n}, \frac{H_{i}(\sqrt{n}t)}{ v_i\sqrt{n}}\right)_{1\leq i\leq d, \ t\geq 0}$ converges in distribution to a continuous process as $n\to\infty$.

By Skorokhod's representation theorem, we construct a sequence of copies 
$(S^n_{\xi_i}, T^{\text{normal},i}_n,$ $H_{n,i})_{1\leq i\leq d}$ and copies of $X_n$ 
on a common probability space such that, for each $1 \leq i \leq d$ and for all $T>0$, 
the following holds almost surely:  
\begin{equation}
\label{eq:DonskerBMi}
\begin{split}
\lim_{n\to\infty}\max_{t\in[0,T]} & \left(\Big|\frac{S^n_{\xi_i}(nt)}{ v_i\sqrt{n}}- W_i(t)\Big| + \Big|\frac{T^{\text{normal},i}_n(nt)}{n}-T^{\text{normal},i}_\infty(t)\Big| \right. \\
&\left. +\Big|\frac{H_{n,i}(\sqrt{n}t)}{ v_i\sqrt{n}}-c_{i}t\Big| \right) = 0,
\end{split}
\end{equation}
\begin{equation}
    \label{eq:sum T}
    \sum_{i=1}^d T^{\text{normal},i}_\infty(t) = t, \quad t\geq 0,
\end{equation}
where $W_1,\dots,W_d$ are independent one-dimensional standard Brownian motions, and $T^{\text{normal},i}_\infty$ are non-decreasing, non-negative continuous processes. More precisely, we apply Skorokhod's theorem to the joint processes $\left(\frac{S_{\xi_i}(nt)}{\sqrt{n}},\frac{T^{\text{normal},i}(nt)}{n}, \frac{H_{i}(\sqrt{n}t)}{ v_i\sqrt{n}}\right)_{1\leq i\leq d, \ t\geq 0}$, then construct the sequence $(X_n)_{n \geq 1}$ on an extended probability space using regular conditional probabilities, ensuring that the representation \eqref{eq:WalshRW_repr} holds. 

From \eqref{eq:estimate_S_Walsh1}, it follows that:
\begin{equation*}\begin{aligned}
    \frac{H_{n,i}(T^{\text{critical}}_n(nt)-1)}{ v_i\sqrt{n}}
 &   \leq -\min_{0\leq k \leq nt}\left(\frac{S^n_{\xi_i}(T^{\text{normal},i}_n(k))}{ v_i\sqrt{n}}\wedge 0\right)  \\
 &   \leq \frac{H_{n,i}(T^{\text{critical}}_n(nt))}{ v_i\sqrt{n}}+ \max_{0\leq k\leq nt}\frac{|\xi_{k+1,i}-\xi_{k,i }|}{ v_i\sqrt{n}}.   
\end{aligned}\end{equation*}
Following the reasoning in Example \ref{expl:Oscillating}, we obtain $\limsup_{n\to\infty}\frac{T^{\text{critical}}_n(nt)}{ \sqrt{n}} < \infty$ a.s. for $t\geq 0$, which yields
\[
    \lim_{n\to\infty} \sup_{t\in[0,T]}\left| \frac{H_{n,i}(T^{\text{critical}}_n(nt))}{ v_i\sqrt{n}} + \min_{0\leq k \leq nt}\left(\frac{S^n_{\xi_i}(T^{\text{normal},i}_n(k))}{ v_i\sqrt{n}}\wedge 0\right)\right| = 0 \quad \text{a.s.}
\]
for every $T>0$. Combining these estimates with \eqref{eq:DonskerBMi}, we obtain the almost sure convergences:
\begin{equation}\begin{aligned}
\label{eq:1143}
\lim_{n\to\infty} \max_{t\in[0,T]}\left|\frac{S_{\xi_i}^n(T^{\text{normal},i}_n(nt))}{ v_i\sqrt{n}} - W_i( T^{\text{normal},i}_\infty(t))\right| = 0, \\
\lim_{n\to\infty} \max_{t\in[0,T]}\left| \frac{H_{n,i}(T^{\text{critical}}_n(nt))}{ v_i\sqrt{n}} - M_i( T^{\text{normal},i}_\infty(t)) \right| = 0,
\end{aligned}\end{equation}
where $M_i(t) \coloneqq -\min_{s\in[0,t]} W_i(s)$. Consequently, 
\begin{align*}
 \lim_{n\to\infty}\max_{t\in[0,T]}\Big|\frac{X_{n,i}(nt)} { v_i\sqrt{n}} - \left( W_i(T^{\text{normal},i}_\infty(t)) + M_i( T^{\text{normal},i}_\infty(t)) \right) \Big| &= 0, \\
 \lim_{n\to\infty}\max_{t\in[0,T]}\Big|\frac{X_{n,i}(nt)} { v_i\sqrt{n}} - W_i^{refl}(T^{\text{normal},i}_\infty(t)) \Big| &= 0.
\end{align*}

From \eqref{eq:1143} and \eqref{eq:DonskerBMi}, there exists a scaling limit for the time spent in critical mode, $\lim_{n\to\infty}\frac{T^{\text{critical}}_n(nt)}{\sqrt{n}} = \frac{M_i( T^{\text{normal},i}_\infty(t))}{c_i}$. Since this is independent of the index $i$, we obtain the balance equation:
\begin{equation}
    \label{eq:balance}
    \frac{M_i( T^{\text{normal},i}_\infty(t))}{c_i} = \frac{M_j( T^{\text{normal},j}_\infty(t))}{c_j}, \quad 1\leq i,j\leq d.
\end{equation}

To finalize the proof, it suffices to observe that the process $(W_i^{refl}(T^{\text{normal},i}_\infty(t)))_{t\geq 0, 1\leq i\leq d}$, where $T^{\text{normal},i}_\infty$ and $M_i$ satisfy \eqref{eq:sum T} and \eqref{eq:balance}, is a Walsh process with parameters $p_i = \frac{c_i}{c_1+\dots+c_d}$, as stated in Theorem \ref{thm:baryaktar}.
\end{proof}

%%%%%%%%%%%%%%%%%%%%%%%%%%%%%%%%%%%%%%%%%%%%%%%%%%%%%

\begin{remark}
We give sufficient conditions that ensure \eqref{eq:LLN_Hi}.
Assume that a.s. there are positive non-random limits 
\begin{align}
\label{eq:LLN_MC}
a_i &\coloneqq \lim_{N\to\infty} ( v_iN)^{-1}\sum_{k=1}^N  |R(\sigma_{k})|{\mathds{1}}_{l(\sigma_{k})=i}, \nonumber \\
b_i &\coloneqq \lim_{N\to\infty} ( v_iN)^{-1}\sum_{k=1}^N   R(\tau_{k }){\mathds{1}}_{l(\tau_{k})=i}, \quad 1\leq  i\leq d.
\end{align}
Then \eqref{eq:LLN_Hi} holds true with $c_i= a_i+b_i, 1\leq i\leq d.$

    % We give   sufficient conditions that ensure  \eqref{eq:LLN_Hi}.
    %  Assume that a.s. there are   positive non-random limits 
    % \begin{align}
    %     \label{eq:LLN_MC}
    %      a_i\coloneqq \lim_{N\to\infty} ( v_iN)^{-1}\sum_{k=1}^N  |R(\sigma_{k})|{\mathds{1}}_{l(\sigma_{k})=i}, \quad  b_i\coloneqq \lim_{N\to\infty} ( v_iN)^{-1}\sum_{k=1}^N   R(\tau_{k }){\mathds{1}}_{l(\tau_{k})=i}, \quad 1\leq  i\leq d.
    % \end{align}
    % Then \eqref{eq:LLN_Hi} holds true with $c_i= a_i+b_i, 1\leq i\leq d.$ 

      Note that $(R^{\text{exit}}(k), l^{\text{exit}}(k))_{k\geq 1}\coloneqq (R(\sigma_k), l(\sigma_k))_{k\geq 1}$ 
and $(R^{\text{entrance}}(k), l^{\text{entrance}}(k))_{k\geq 1}\coloneqq\  $  $  (R(\tau_k), l(\tau_k))_{k\geq 1} = 
(R(\sigma_k+1), l(\sigma_k+1))_{k\geq 1}$ are Markov chains. It is known, e.g. \cite[p. 126]{IPMS_book}, that
\begin{equation}\begin{aligned}\label{eq: RW to 0}
\inf_{x\in {\mathbb N}}{\mathop{\rm P}}_x\big(S_\xi(\sigma_{{\mathbb Z}\setminus {\mathbb N}})=0\big)>0,
\end{aligned}\end{equation} 
for any 1-arithmetic one-dimensional integer-valued  random walk $S_\xi$ with ${\mathop{\rm E}}\xi=0, {\mathop{\rm Var}} \, \xi \in(0,\infty)$,
where $\sigma_{{\mathbb Z}\setminus {\mathbb N}}\coloneqq \inf\{n\geq 0 : \ S_\xi(n)\leq0\}.$ Since we assume that all states  ${\mathbb N}\times\{1,\dots,d\}$ communicate  for the chain $(R,l)$, formula \eqref{eq: RW to 0} implies that Markov chains $(R^{\text{exit}}, l^{\text{exit}})$ and $(R^{\text{entrance}}, l^{\text{entrance}})$ possess unique invariant  probability distributions $\pi^{\text{exit}}$ and $\pi^{\text{entrance}}$. 
 
By the law of large numbers for Markov chains, almost sure limits  in  \eqref{eq:LLN_MC}  exists if  the first moments of $R$ under  $\pi^{\text{exit}}$ and $\pi^{\text{entrance}}$ are finite.  In this case  parameters of the Walsh process are
\begin{equation}\begin{aligned}\label{eq:param_Walsh1}
p_i=  \frac{  v_i^{-1}({\mathop{\rm E}}_{\pi^{\text{exit}}}|R^{\text{exit}}(0)|{\mathds{1}}_{l^{\text{exit}}(0)=i}+{\mathop{\rm E}}_{\pi^{\text{entrance}}}R^{\text{entrance}}(0){\mathds{1}}_{l^{\text{entrance}}(0)=i}) }{\sum_{j=1}^m   v_j^{-1}({\mathop{\rm E}}_{\pi^{\text{exit}}}|R^{\text{exit}}(0)|{\mathds{1}}_{l^{\text{exit}}(0)=j}+{\mathop{\rm E}}_{\pi^{\text{entrance}}}R^{\text{entrance}}(0){\mathds{1}}_{l^{\text{entrance}}(0)=j})}.
\end{aligned}\end{equation}
Sufficient condition  ensuring ${\mathop{\rm E}}_{\pi^{\text{exit}}}|R^{\text{exit}}(0)|+{\mathop{\rm E}}_{\pi^{\text{entrance}}}R^{\text{entrance}}(0)<\infty$ is, for example, the following
\begin{equation}
    \label{eq:ass_pavl}
    \exists C>0\quad \forall i,\ 1\leq i\leq m\quad \forall x\in{\mathbb Z}\setminus {\mathbb N}\  \ \quad 
{\mathop{\rm E}}_{(R(0), l(0))=(x,i)}(R(1) )\leq C(1+|x|),
\end{equation}
see \cite{PavlyukevichPilipenko2023, IPMS_book}.
\end{remark}

%%%%%%%%%%%%%%%%%%%%%%%%%%%%%%%%%%%%%%%%%%%%%%%%%%
 
%%%%%%%%%%%%%%%%%%%%%%%%%%%%%%%%%%%%%%%%%%%%%

\begin{example}
    Let $X$ be a perturbed random walk on ${\mathbb Z}$ with a membrane $A=\{-m,\dots,m\}$ and transition probabilities 
    \begin{equation*}
        p_{x,y}=\begin{cases}
        {\mathop{\rm P}}(\xi=y-x), & x\notin A;\\
        {\mathop{\rm P}}(\eta_x=y), & x\in A,
    \end{cases}
    \end{equation*}
    where ${\mathop{\rm E}}\xi=0, {\mathop{\rm Var}} \, \xi =\sigma^2\in(0,\infty),$ and ${\mathop{\rm E}}|\eta_x|<\infty$ for $x\in A$. Without loss of generality, we assume that ${\mathop{\rm P}}(X(n+1)\notin A \mid X(n)=x)=1$ for $x\in A$ (see Lemma \ref{lem:equivalent perturbed with skips} and Remark \ref{remk:skips}), which means ${\mathop{\rm P}}(| \eta_x|>m)=1$ for all $x \in A$. We associate with $X$ an auxiliary Markov chain $(R,l)$ on ${\mathbb Z}\times \{-,+\}$ that satisfies the assumptions of Theorem \ref{thm:pert_Walsh}.

    The transition probabilities of $(R,l)$ are defined as follows:
    If $x\in {\mathbb N}$, then
    \begin{equation}
        {\mathop{\rm P}}\Big((R(n+1), l(n+1))=(x+y,i) \mid (R(n), l(n))=(x,i)\Big)=\begin{cases}
            {\mathop{\rm P}}(\xi=y), & i=+, \\
            {\mathop{\rm P}}(\xi=-y), & i=-.
        \end{cases}
    \end{equation}
    If $x\leq -2m-1$, the particle jumps over the membrane area:
    \begin{equation}
        \begin{aligned}
            {\mathop{\rm P}}\Big((R(n+1), l(n+1))=(-2m-x,-) \mid (R(n), l(n))=(x,+)\Big) &= 1, \\
            {\mathop{\rm P}}\Big((R(n+1), l(n+1))=(-2m-x,+) \mid (R(n), l(n))=(x,-)\Big) &= 1.
        \end{aligned}
    \end{equation}
    And if $-2m\leq x\leq 0$   and $y\geq 0$, then:
    \begin{equation}
        \begin{aligned}
            {\mathop{\rm P}}\Big((R(n+1), l(n+1))=(y,+) \mid (R(n), l(n))=(x,+)\Big) &= {\mathop{\rm P}}(\eta_{x+m}=y+m), \\
            {\mathop{\rm P}}\Big((R(n+1), l(n+1))=(y,-) \mid (R(n), l(n))=(x,+)\Big) &= {\mathop{\rm P}}(\eta_{x+m}=-y-m), \\
            {\mathop{\rm P}}\Big((R(n+1), l(n+1))=(y,-) \mid (R(n), l(n))=(x,-)\Big) &= {\mathop{\rm P}}(\eta_{-x-m}=-y-m), \\
            {\mathop{\rm P}}\Big((R(n+1), l(n+1))=(y,+) \mid (R(n), l(n))=(x,-)\Big) &= {\mathop{\rm P}}(\eta_{-x-m}=y+m).
        \end{aligned}
    \end{equation}

    Define the map $\phi:{\mathbb Z}\times \{-,+\}\to{\mathbb Z}$ as
    \[
    \phi(r,i)=i(r-m)=\begin{cases}
        (r-m), & i=+,\\
        -(r-m), & i=-.
            \end{cases}
    \]
    Note that $\phi:{\mathbb N} \times \{-,+\} \to{\mathbb Z}\setminus A$ is a bijection. 
    %rata
    Consider the sequence $\bar X(n)\coloneqq \phi(R(n),l(n)), n\geq 0$. By construction, $\bar X$ satisfies:
    \begin{itemize}
     \item If $(\bar X(n-1) >m, \bar X(n ) >m ) $ or $(\bar X(n-1) <-m, \bar X(n )<-m)   $  or  $(\bar X(n-1) \in A, |\bar X(n )|\notin A),   $  then $\bar X(n+1)-\bar X(n)\overset{d}{=} \xi$.
       \item  If $\bar X$ jumps over $A$ (i.e., $\bar X(n-1) >m, \bar X(n ) <-m  $ or $\bar X(n-1) <-m, \bar X(n )>m$), it remains stationary for one step ($\bar X(n+1)=\bar X(n)$).  
        \item If $ \bar X(n )=x\in A,$ then   $\bar X(n+1) \overset{d}{=} \eta_x$. 
                   \end{itemize}

These properties imply that the jumps of $\bar{X}$ follow the distribution of $\xi$ until $\bar{X}$ either hits the membrane $A$ or jumps over it. If $\bar{X}$ hits a state $x \in A$, the subsequent jump follows the distribution of $\eta_x$, after which the increments of $\bar{X}$ return to the distribution of $\xi$ until the next encounter with $A$. If $\bar{X}$ jumps over $A$, it remains stationary for exactly one step, and subsequently continues jumping with the distribution of $\xi$ until it hits or jumps over $A$ again.

If  $\bar X(0)\overset{d}{=} X(0) $ and  we remove the idle moments after jumps over   $A$, the resulting sequence will have the same distribution as $X$. Similarly to representations considered in Lemma \ref{lem:equivalent perturbed with skips} the last sentence means \[
(\bar X  (\bar T^{-1}(n)-1))_{n\geq 0} \overset{d}{=} (X(n))_{n\geq 0},
\]
where $\bar T(n)=n-\sum_{k=0}^{n-1}\left({\mathds{1}}_{\bar X(k-1)>m, \bar X(k)<-m }+{\mathds{1}}_{\bar X(k-1)<-m, \bar X(k)>m }\right).$

According to Theorem \ref{thm:pert_Walsh}, the processes $\left(\frac{R(nt){\mathds{1}}_{l(nt)=+}}{\sigma\sqrt{n}}, \frac{R(nt){\mathds{1}}_{l(nt)=-}}{\sigma\sqrt{n}}\right)_{t\geq 0}$ converge in distribution to a two-dimensional Walsh process $(W_+(t), W_-(t))_{t\geq 0}$ with certain parameters $p_\pm$.
     Thus, $(W_+(t)- W_-(t))_{t\geq 0}$ is a skew Brownian motion with parameter $\gamma=p_+-p_-$. Consequently, $\big(\frac{\bar X(nt)}{\sqrt{n}}\big)_{t\geq 0}$ converges to this skew Brownian motion. Since $\big(\frac{\bar T (nt)}{n}\big)_{t\geq 0}$ converges locally uniformly to $(t)_{t\geq 0}$ a.s., the convergence of $\big(\frac{ X(nt)}{\sqrt{n}}\big)_{t\geq 0}$ to the same limit follows.
\end{example}

\textbf{Discussion and possible generalizations.}

\begin{itemize}
    \item We never used in the proof the fact that $X$ is an integer-valued Markov chain. It seems plausible to assume that \eqref{eq:ass_pavl} and the absolute continuity of the distributions of $\xi_i, 1\leq i\leq d,$ imply the existence of limits in \eqref{eq:LLN_Hi}.

    \item Consider a collection of sequences $(R^{(n)}(k),l^{(n)}(k))_{k\geq 0}$ that have the same transition probabilities as $(R(k), l(k))_{k\geq 0}$, but may have different initial conditions. Construct $(X^{(n)}(k))_{k\geq 0}$ similarly to \eqref{eq:R-to_x}. Assume that $X^{(n)}(0) \xrightarrow{{\mathop{\rm P}}} x \in E.$ Notice that the distribution of entrance and exit Markov chains depends on $n$, requiring a modification of \eqref{eq:LLN_Hi_uniform}. It was proved in \cite{PavlyukevichPilipenko2023} that \eqref{eq:ass_pavl} is sufficient for the convergence of scaling limits to the Walsh process started from $x$. The idea was as follows: the scaling of the random walk stopped at the first instant of crossing 0 converges to the Brownian motion on the corresponding ray stopped at hitting 0. After the first crossing, formula \eqref{eq: RW to 0} ensures mixing conditions and the law of large numbers. 

    \item Assume a collection of sequences $(R^{(n)}(k),l^{(n)}(k))_{k\geq 0}$ is such that their transition probabilities depend on $n$ and the corresponding random variables $\xi_i^{(n)}$ satisfy
    \[
    \Big(\frac{S_{\xi_1^{(n)}}(n\cdot)}{ v_1\sqrt{n}},\dots,\frac{S_{\xi_d^{(n)}}(n\cdot)}{ v_d\sqrt{n}}\Big) \Rightarrow (W_1,\dots, W_d), \quad n\to\infty,
    \]
    where $W_1,\dots, W_d$ are independent Wiener processes. We also assume that $X^{(n)}(0) \xrightarrow{{\mathop{\rm P}}} x \in E.$ It would be interesting to find a wide class of sufficient conditions assuring convergence to a Walsh process. We stress that almost all our reasoning was based on Skorokhod's representation theorem, non-random convergence results, and the characterization of a Walsh Brownian motion as in Theorem \ref{thm:baryaktar}. The only important points of the proof were the law of large numbers \eqref{eq:LLN_Hi_uniform} or \eqref{eq:LLN_MC}, and checking the small amount of time spent in the non-positive parts of the lines, see \eqref{eq:crit-to0}. Finding counterparts of these conditions may be a much simpler question than proving a functional limit theorem for RW on graphs.

    \item The proof of Theorem \ref{thm:pert_Walsh} presented in this paper is new. The method is based on the approach proposed in \cite{ChulakovPilipenko}, where deterministic switching systems and a generalized Skorokhod map on a graph were proposed. The previous proof of the theorem (under condition \eqref{eq:ass_pavl}), see \cite{PavlyukevichPilipenko2023, PavlyukevichPilipenko2022,  IPMS_book}, was based on a characterization of the martingale problem of a Walsh process, see Theorem \ref{thm:Walsh}. It was verified that limit points of the sequence $(\frac{S_{\xi_i}(T^{\text{normal},i}(nt))}{ v_i\sqrt{n}})$ are continuous martingales; however, the verification of \eqref{e:brM} and the association of the process $L$ with a limit of $(\sum_{i=1}^d\frac{H_{i}(T^{\text{critical}}(nt))}{ v_i\sqrt{n}})$ were technically demanding.

    There are other methods for the investigation of perturbed random walks in particular cases. For example, in \cite{NgoPeigne, Vo+Peigne:2023}, the classical approach (based on weak relative compactness and the convergence of finite-dimensional distributions) was applied. These works consider an oscillating RW on integers where the membrane consists of a single point. In the absence of explicit formulas for transition probabilities, their study required elegant methods combining renewal theory, local limit theorems, and other tools. However, extending these techniques to the general case may encounter significant technical difficulties. For a semigroup approach to the study of perturbed random walks, see \cite{MinlosZhizhina}. Other probabilistic methods, based on couplings, excursions, and similar techniques, have been explored in \cite{EnriquesKifer, PilipenkoPrykhodkoUMZh, pilipenkosakhanenko2015limit, LambertSimatos, mijatovic2022limit, Yano2008convergence, Yano2015functional}.

    \item Let us discuss difficulties that may arise in the proof of Theorem \ref{thm:pert_Walsh} if we assume that jumps from the non-positive parts of half-axes are heavy-tailed. Analogously to \eqref{eq:LLN_Hi}, assume that there is a sequence $(b_n)$ such that we have convergence in distribution:
  % \begin{equation}\begin{aligned}
% \label{eq:wiener+subord}
% \left(\frac{S_{\xi_1}(nt)}{ v_1\sqrt{n}} ,\dots,\frac{S_{\xi_d}(nt)}{ v_d\sqrt{n}},\frac{H_1( b_nt)}{ v_1\sqrt{n}},\dots, \frac{H_d( b_nt)}{ v_d\sqrt{n}} \right)_{t\geq 0}\Rightarrow
% \left(W_1(t),\dots,W_d(t),V_1(  t),\dots,V_d(  t) \right)_{t\geq 0}, \quad n\to\infty,
% \end{aligned}\end{equation}
% where all coordinates are independent,     $V_i$ are non-degenerate $\alpha$-stable subordinators with $\alpha \in (0,1)$.
\begin{equation}
\label{eq:wiener+subord}
\Biggl( \left(\frac{S_{\xi_i}(nt)}{v_i\sqrt{n}}, \, \frac{H_i(b_nt)}{v_i\sqrt{n}}\right)_{\!1 \leq i \leq d} \,\Biggr)_{t \geq 0} 
\Rightarrow \Biggl( \left( W_i(t), \, V_i(t) \right)_{1 \leq i \leq d} \Biggr)_{t \geq 0}, \quad n\to\infty,
\end{equation}
  where all coordinates are independent, and $V_i$ are non-degenerate $\alpha$-stable subordinators with $\alpha \in (0,1)$.

   Generally, the Markov chains $(X_i(\sigma_n))$ and $(X_i(\tau_n))$ for $1\leq i\leq d$ and $n\in{\mathbb N}$ are dependent and depend on the sequences $(\xi_{n,i})_{n\in{\mathbb N}}, 1\leq i\leq d$. Consequently, the convergence in \eqref{eq:wiener+subord} requires additional arguments. For example, condition \eqref{eq:wiener+subord} is satisfied if the random variables $\xi_i$ are bounded for $1\leq i\leq d$ and 
\[
{\mathop{\rm P}}_{(R(0), l(0))=(x,i)}(R(1){\mathds{1}}_{l(1)=j}>y)\sim \frac{c_{x,i,j}}{y^\alpha \varphi(y)}, \quad y\to+\infty, \ x\leq0, \ 1\leq i,j\leq d,
\]
where $c_{x,i,j}$ are positive constants and $\varphi$ is a function slowly varying at infinity. It would be interesting to identify suitable additional sufficient assumptions on $c_{x,i,j}$ for the case where $\xi_i$ are unbounded, analogous to \eqref{eq:ass_pavl}.

%rata
As in the proof of Theorem \ref{thm:pert_Walsh}, we can construct a new probability space and copies of $S_{\xi_i}$ and $H_i$ such that their scalings converge with probability 1 to a copy of $(W_1,\dots,W_d,V_1,\dots,V_d)$ %which will be
denoted using the same symbols. We can also assume that copies of $X$ are defined on this probability space and, by passing to subsequences if necessary, that we have almost sure limits $\lim_{n\to\infty}\frac{T^{\text{normal},i}_n(nt)}{n}=T^{\text{normal},i}_\infty(t)$, where the continuous processes $(T^{\text{normal},i}_\infty)$ satisfy \eqref{eq:sum T}. 

Proceeding as in Example \ref{expl:refl_Levy} and Theorem \ref{thm:pert_Walsh}, inequality \eqref{eq:estimate_S_Walsh1} and the convergence $\lim_{n\to\infty}\frac{T^{\text{normal},i}_n(nt)}{n}=T^{\text{normal},i}_\infty(t)$ imply that $\lim_{n\to\infty} \frac{T^{\text{critical}}_n(nt)}{b_n}=L(t)$, where $L(t)$ is the unique number satisfying the inequality
\begin{equation}
    \label{eq:ineq_H_i_LM}
    V_i(L(t)-)\leq M_i(T^{\text{normal},i}_\infty(t))\leq V_i(L(t)),
\end{equation}
cf. \eqref{eq:ineq_M_H}. Here $M_i(t)=-\min_{s\in[0,t]}W_i(s)$. Hence, 
\begin{equation}
    \label{eq:eq for L}
     L(t)= V_i^{-1}\circ M_i(T^{\text{normal},i}_\infty(t)), \quad t\geq0, \ 1\leq i\leq d.
\end{equation}
Equation \eqref{eq:eq for L} represents a system for the unknown processes $T^{\text{normal},i}_\infty$ and $L$ that satisfy \eqref{eq:sum T}. It follows from \cite[Lemma 6.2]{Bobrowski-Pilipenko_2025_Walsh} that this system has a unique solution for all $\omega$ such that the processes $V_i^{-1}\circ M_i$ do not share the same intervals of constancy. This condition is satisfied with probability 1 (see the uniqueness proof in Theorem 6.1 of \cite{Bobrowski-Pilipenko_2025_Walsh}), since $(V_i^{-1}\circ M_i)^{-1}= M_i\circ V_i^{-1}$ are independent subordinators. The explicit formula for $L$ and $T^{\text{normal},i}_\infty$ in terms of $M_i$ and $V_i$ is provided in \cite[\S 10]{Bobrowski-Pilipenko_2025_Walsh}. These results are analytical and non-random; they do not depend on specific properties of  stochastic processes.

If we can show that 
\begin{equation}\label{eq:1433}
    \lim_{n\to\infty} \frac{H_{n,i}(b_n\frac{T^{\text{critical}}_n(nt)}{b_n})}{v_i \sqrt{n}}= \lim_{n\to\infty} \frac{H_{n,i}(b_n \cdot )}{v_i \sqrt{n}}\circ \frac{T^{\text{critical}}_n(nt)}{b_n}=V_i(L(t)), \quad 1\leq i\leq d,
\end{equation}
almost surely in ${\mathcal D}$, then we obtain the convergence 
\[
\lim_{n\to\infty} X_{n,i}(t)=W_i(T^{\text{normal}, i}_\infty(t))+V_i(L(t))= (W_i+V_i\circ V_i^{-1}\circ M_i)\circ T^{\text{normal}, i}_\infty(t), \quad 1\leq i\leq d,
\]
almost surely in ${\mathcal D}$. We conjecture that \eqref{eq:1433} holds without additional assumptions; the proof would likely be similar to the $d=1$ case in \cite{Pilipenko+Sarantsev:2024}. This is the most challenging part, as composition is not a continuous map in ${\mathcal D}$, and a concise proof or reference is currently unavailable.

The expression on the right-hand side indicates that the $i$-th coordinate behaves like the process from Example \ref{expl:refl_Levy} in an ``inner time'' $T^{\text{normal}, i}_\infty$. It can be shown that for almost all $t\geq 0$, only one process among $T^{\text{normal}, i}_\infty, 1 \leq i \leq d,$ is changing. 

The process $\big((W_i+V_i\circ V_i^{-1}\circ M_i)\circ T^{\text{normal}, i}_\infty\big)_{1\leq i\leq d}$ is a Feller process on a graph $E_d\subset {\mathbb R}^d$ consisting of the positive coordinate axes. For a comprehensive description of Brownian motions on star graphs, we refer to \cite{Bobrowski-Pilipenko_2025_Walsh}. Specifically, the generator of the limit process is one-half of the second derivative along the axes, restricted to the set of twice continuously differentiable functions $f$ on the graph that are bounded with their second derivative and satisfy the boundary condition at 0:
\[
\int_{E_d}(f(x)-f(0)) m(\mathrm{d} x)=0.
\]
Here $m$ is the Lévy measure of $(V_1,\dots,V_d)$. Note that $m$ is supported on $E_d$, as the $\alpha$-stable subordinators $V_1,\dots,V_d$ are assumed to be independent.\end{itemize}

\section{Skew stable Lévy process. Heavy tailed random walks and perturbations.}\label{sec:skew Levy} 

Unlike local perturbations of Brownian motion, such as reflected Brownian motion, skew Brownian motion, and the Walsh process, perturbations of  stable Lévy processes have been studied much less extensively.   Progress in this area was made relatively recently in article \cite{IksanovPilipenko2021skewLevy}, where a natural construction of a  skew symmetric stable Lévy process was obtained as a limit of local perturbations of the Lévy process.  
The process was described in three ways: (a) through an explicit form of the resolvent, (b) via It\^o's excursion theory and   the entrance law of the  process, (c) by describing  a stochastic differential equation of type \eqref{eq:guessLimit}, where $L$ denotes the  Blumenthal-Getoor local time.  

This section presents a theorem on the convergence of perturbed random walks to a skew stable Lévy process when the membrane is a single point (see \cite{Dong+Iksanov+Pilipenko:2024+}), along with proof ideas and potential generalizations. Details on points (b) and (c) are omitted due to the extensive background required from Itô's excursion theory and its technical complexity.

Before discussing  definitions and   results,  note that SDEs involving a local time and  Lévy noise are rarely  studied, unlike those for Brownian noise.   In fact, even the simple  equation \eqref{eq:skew_SDE_Z}, where a symmetric stable process $U_\alpha$ with $\alpha\in(1,2)$ replaces the Brownian motion $W$, there are  no solutions in the class of Feller processes unless $\gamma= 0$.    This observation is related to the following circumstance, based on It\^o's excursion theory (see precise statements  and details in   \cite[Theorems 2.3.3, 2.3.4]{IPMS_book}).  Specifically, any solution to  \eqref{eq:skew_SDE_Z} is a Markov extension of the process $U_\alpha$ stopped   at 0  with a continuous entrance law (from $0$). However, the  only one non-trivial honest extension  without positive sojourn at 0  that also has a continuous entrance law  is $U_\alpha$ itself.

For a real-valued strong Markov process $U$, denote by $R^{U}_\lambda$ and $V^{U}_\lambda$ the resolvents of $U $ and the processes $U
$ killed upon hitting $0$, respectively, that is, 
\[
R^{U  }_\lambda f(x)\coloneqq {\mathop{\rm E}}_x\int_0^\infty  {\rm e}^{-\lambda t} f(U(t)){\rm d}t,\quad \lambda>0
\]
and  
\[
V^U_\lambda f(x)\coloneqq {\mathop{\rm E}}_x\int_0^{\sigma_0(U )} {\rm e}^{-\lambda t} f(U(t)){\rm d}t,\quad \lambda>0,
\]
where $f$ is a bounded measurable function,   $\sigma_0(U)\coloneqq  \inf\{ t\geq 0\ : \ U
(t)=0\}.$

 Consider   a perturbed  RW $X$ on ${\mathbb Z}$ with a membrane $A\coloneqq \{0\}$ having  transition probabilities 
    \begin{equation}\label{eq:transitions_PRW_Levy}
    p_{x,y}=\begin{cases}
    {\mathop{\rm P}}(\xi=y-x), & x\neq 0;\\
    {\mathop{\rm P}}(\eta =y), & x=0.
\end{cases}
\end{equation}
To avoid trivialities we will assume that all elements of ${\mathbb Z}$ communicate.  Assume that  the random variable $\xi$ belongs to domain of attraction of  symmetric $\alpha$-stable distribution with $\alpha\in(1,2)$. It is well known, e.g., Theorem 2.7 in \cite{Skorokhod:1957}, that there is a sequence $(a_n)$ such that
\begin{equation}\label{eq:conv}
\Big(\frac{S_\xi(nt)}{a_n}\Big)_{t\geq 0}~\Rightarrow~ \Big(U_\alpha(t)\Big)_{t\geq 0},\quad n\to\infty,
\end{equation}
in ${\mathcal D}$, where $U_\alpha$ is a symmetric  $\alpha$-stable process, ${\mathop{\rm E}} {\rm e}^{iz U_\alpha(t)}=\exp(-t|z|^\alpha), z\in{\mathbb R}.$
We will also assume that $\eta$ either has a finite expectation  or $\eta$ belongs to the domain of attraction of $\beta$-stable law with $\beta\in(0,1),$ i.e., there is a sequence $(b_n)$ such that $\Big(\frac{S_\eta(n)}{b_n}\Big)$ converges in distribution to a $\beta$-stable law with the Lévy measure $m(\mathrm{d} y)=
(c_-{\mathds{1}}_{(-\infty,\, 0)}(y)+c_+{\mathds{1}}_{(0,\,\infty)}(y))|y|^{-(1+\beta)}{\rm d}y,\quad y\in\mathbb{R}$, where
   $c_\pm $ are non-negative constants, $c_++c_->0$.

\begin{theorem}\label{thm:lim_skew_Levy_walk}

\noindent  (a) 
  If ${\mathop{\rm E}} |\eta|<\infty$ or the distribution of $\eta$ belongs to the domain of attraction of a $\beta$-stable distribution with $\beta>\alpha-1$, 
then 
\begin{equation}\label{eq:FCLT2}
\Big(\frac{X(nt)}{a_n}\Big)_{t\geq 0}~\Rightarrow~ \big(U_\alpha(t)\big)_{t\geq 0},\quad n\to\infty,
\end{equation}

\noindent  (b) 
If the distribution of $\eta$ belongs to the domain of attraction of a $\beta$-stable distribution with $\beta<\alpha-1$, then 
sequence $\Big(\frac{X(nt)}{a_n}\Big)_{t\geq 0} $ converges in distribution
 in ${\mathcal D}$ as $n\to\infty$ to a Feller process $U_{\alpha,\beta}$ started from $0$ whose resolvent is equal to
\begin{equation}\label{eq:RES}
R^{U_{\alpha,\beta}}_\lambda f(x) =V^{U_\alpha}_\lambda f(x)+\frac{\int_{{\mathbb R}}V^{U_\alpha}_\lambda f \; \mathrm{d} m }{\lambda \int_{{\mathbb R}}V^{U_\alpha}_\lambda 1\; \mathrm{d} m }{\mathop{\rm E}}_x {\rm e}^{-\lambda \sigma(U_\alpha)},\quad \lambda>0,\ x\in{\mathbb R}, % \frac{\int_{{\mathbb R}}V_\lambda f(x)\eta^\ast(dx)}{\lambda \int_{{\mathbb R}}V_\lambda 1 (x)\eta^\ast(dx)} 
\end{equation}
\end{theorem} 

The fact that \eqref{eq:RES} is indeed a resolvent can be found in \cite{IksanovPilipenko2021skewLevy} or \cite[\S 2.3]{IPMS_book}. We call this process a skew symmetric Lévy process with parameters $\alpha\in(1,2), \beta\in(0,\alpha-1).$
 It is known that $U_{\alpha,\beta}$ satisfies \eqref{eq:guessLimit}, where $L$ is the Blumenthal-Getoor local time, $U=U_\alpha$ and $V$ is $\beta$-stable process whose Lévy measure equals $C m(\mathrm{d} x)$ and $C=\frac{\beta\sin\frac{\pi(1+\beta)}{\alpha}}{(c_-+c_+)\Gamma(1-\beta) \cos\frac{\pi\beta}{2}\sin\frac{\pi}{\alpha}},$ see \cite[Corollary 2.3.1]{IPMS_book}. Moreover, the distributions of $U_\alpha$ and $U_{\alpha,\beta}$ up to the hitting 0 are identical.
 \begin{remark}
     The critical case $\beta = \alpha - 1$ was not addressed in \cite{IPMS_book, IksanovPilipenko2021skewLevy}. We have substantial arguments supporting the conjecture that if $\beta = \alpha - 1$, the limit is an unperturbed stable process $U_\alpha$. Unfortunately, since we currently have a rigorous proof for this fact only in certain specific cases, this issue is not considered in the present paper.
 \end{remark}
 \begin{proof}[Idea of the proof of Theorem \ref{thm:lim_skew_Levy_walk} (a)]
The alternative representation \eqref{eq:20_representation_PRW} takes the form:$$Y(n)=  S_\xi( T^{\text{normal}}(n))+  S_{\eta }(T^{\text{critical}}(n)).$$ Because $\xi$ belongs to the domain of attraction of an $\alpha$-stable law with $\alpha \in (1,2)$, the random variable $\xi$ is   zero mean, which leads to the convergence established in \eqref{eq:conv-stab-proc}. Consequently, the proof of \eqref{eq:FCLT2} reduces to showing that the "critical" term vanishes in the limit:$$\Big(\frac{S_\eta(T^{\text{critical}}(nt)  )}{a_n}\Big)_{t\geq 0} \Rightarrow \Big(0\Big)_{t\geq 0}, \quad n\to\infty.$$By applying the bound $\max_{0\leq k\leq N}|S_\eta(k)|\leq S_{|\eta|}(N)$, we see that it is sufficient to verify the following convergence in probability:\begin{equation}\label{eq:S_eta_small}\frac{S_{|\eta|}(T^{\text{critical}}(nT)  )}{a_n} \overset{{\mathop{\rm P}}}{\to} 0, \quad n\to\infty,\end{equation}for any $T>0.$ Finally, noting that $(a_n)$ is regularly varying at infinity with parameter $1/\alpha,$ we may assume $T=1$ without loss of generality.
 
      If $\eta$ belongs to domain of attraction of $\beta$-stable law, then   
     \begin{equation}
          \forall \delta>0\quad \frac{S_{|\eta|}(n )}{n^{\frac{1}{\beta}+\delta}} \Rightarrow 0, \quad n\to\infty.
          \label{eq:|eta|small}
     \end{equation}
For any $\varepsilon_1, \varepsilon_2>0$ and sufficiently large $n$ we have
%rata
% \begin{equation}
%     \begin{aligned}
%         \frac{S_{|\eta|}(T^{\text{critical}}(n )  )}{a_n}\leq \frac{S_{|\eta|}(T^{\text{critical}}(n )  )}{n^{\frac{1}{\alpha}-\varepsilon_2}}
%   \leq \frac{S_{|\eta|}(T^{\text{critical}}(n )  )}{n^{\frac{1}{\alpha}-\varepsilon_2}} {\mathds{1}}_{T^{\text{critical}}(n)\geq n^{\frac{\beta(1-\varepsilon_1)}{\alpha}}}+
%   \frac{S_{|\eta|}( n^{\frac{\beta(1-\varepsilon_1)}{\alpha}})}{n^{\frac{1}{\alpha}-\varepsilon_2}} {\mathds{1}}_{T^{\text{critical}}(n)\leq n^{\frac{\beta(1-\varepsilon_1)}{\alpha}}} \\
%   \leq \frac{S_{|\eta|}(T^{\text{critical}}(n )  )}{n^{\frac{1}{\alpha}-\varepsilon_2}} {\mathds{1}}_{T^{\text{critical}}(n)\geq  n^{\frac{\beta(1-\varepsilon_1)}{\alpha}}}+  \frac{S_{|\eta|}( n^{\frac{\beta(1-\varepsilon_1)}{\alpha}})}{ n^{\frac{ 1-\frac{\varepsilon_1}{2}}{\alpha} }} \frac{ n^{\frac{ 1-\frac{\varepsilon_1}{2}}{\alpha} } }{n^{\frac{1}{\alpha}-\varepsilon_2}} 
%     \end{aligned}
% \end{equation}

\begin{equation}
\label{eq:S_eta_bounds}
\begin{aligned}
\frac{S_{|\eta|}(T^{\text{critical}}(n))}{a_n}
&\leq \frac{S_{|\eta|}(T^{\text{critical}}(n))}{n^{\frac{1}{\alpha}-\varepsilon_2}} \\
&\leq \frac{S_{|\eta|}(T^{\text{critical}}(n))}{n^{\frac{1}{\alpha}-\varepsilon_2}} {\mathds{1}}_{\left\{T^{\text{critical}}(n) \geq n^{\frac{\beta(1-\varepsilon_1)}{\alpha}}\right\}} \\
&\quad + \frac{S_{|\eta|}(n^{\frac{\beta(1-\varepsilon_1)}{\alpha}})}{n^{\frac{1}{\alpha}-\varepsilon_2}} {\mathds{1}}_{\left\{T^{\text{critical}}(n) \leq n^{\frac{\beta(1-\varepsilon_1)}{\alpha}}\right\}} \\
&\leq \frac{S_{|\eta|}(T^{\text{critical}}(n))}{n^{\frac{1}{\alpha}-\varepsilon_2}} {\mathds{1}}_{\left\{T^{\text{critical}}(n) \geq n^{\frac{\beta(1-\varepsilon_1)}{\alpha}}\right\}} \\
&\quad + \frac{S_{|\eta|}(n^{\frac{\beta(1-\varepsilon_1)}{\alpha}})}{n^{\frac{1-\frac{\varepsilon_1}{2}}{\alpha}}} \frac{n^{\frac{1-\frac{\varepsilon_1}{2}}{\alpha}}}{n^{\frac{1}{\alpha}-\varepsilon_2}}
\end{aligned}
\end{equation}

If $\varepsilon_1 > 2\varepsilon_2$, the second summand on the right-hand side converges to zero in probability. Hence, \eqref{eq:S_eta_small} holds if we can find $\varepsilon_1$ such that:
\[\frac{T^{\text{critical}}(n)}{ n^{\frac{\beta(1-\varepsilon_1)}{\alpha}}}\overset{{\mathop{\rm P}}}{\to}0,\quad n\to\infty\]
for some $\varepsilon_1>0.$ Given that $\beta > \alpha - 1$, it is sufficient to prove the existence of $\varepsilon > 0$ such that:
\[
\frac{T^{\text{critical}}(n)}{ n^{1-\frac{1}{\alpha}+\varepsilon}}\overset{{\mathop{\rm P}}}{\to}0,\quad n\to\infty.
\]
Let $\tau_n$ be $n$th arrival of $X$ to 0. Then $\theta_n\coloneqq \tau_{n+1}-\tau_n, n\in{\mathbb N},$ are independent identically distributed random variables.
For any $C>0$
     \[
     {\mathop{\rm P}}(\frac{T^{\text{critical}}(n)}{ n^{1-\frac{1}{\alpha}+\varepsilon}}>C)\leq {\mathop{\rm P}}(S_{\theta}([C n^{1-\frac{1}{\alpha}+\varepsilon}])<n).
     \]
     For any $x\in{\mathbb Z}$ and $N\geq 1$
     \[
     {\mathop{\rm P}}(\theta>N)\geq {\mathop{\rm P}}(\eta=x){\mathop{\rm P}}_x(\sigma_0(S_\xi)\geq N)= {\mathop{\rm P}}(\eta=x){\mathop{\rm P}}_0(\sigma_0(S_\xi)\geq N)\frac{{\mathop{\rm P}}_x(\sigma_0(S_\xi)\geq N)}{{\mathop{\rm P}}_0(\sigma_0(S_\xi)\geq N)},
     \]
where $\sigma_0(S_\xi)\coloneqq \inf\{k\in{\mathbb N}  \ : \ S_\xi(k)=0\}.$
     
 By Theorem T1 on p.~378 in \cite{Spitzer}, 
\begin{equation}\label{eq:ratio}
\lim_{N\to\infty} \frac{{\mathop{\rm P}}_x(\sigma_0(S_\xi)\geq N)}{{\mathop{\rm P}}_0(\sigma_0(S_\xi)\geq N)}=g(x)\in [0,\infty),
\end{equation}
where $g$ is the potential kernel of $S_\xi$. By Theorem P2 on p.~361 of the same reference, $g(x)>0$ for all $x\in\mathbb{Z}\backslash\{0\}$. According to Lemma 2.1 in \cite{Belkin:1970}, ${\mathop{\rm P}}(\sigma_0(S_\xi)>N)~\sim~ N^{-(1-1/\alpha)}L(N )$ as $N\to\infty$ for some $L$ slowly varying at $\infty$. Therefore for any $\widetilde \varepsilon>0$ there is $c>0$ such that
\begin{equation}
     {\mathop{\rm P}}(\theta>N)\geq   \frac{c}{ N^{1-1/\alpha+\widetilde \varepsilon}},\quad N\geq 1.
     \label{eq: long return}
\end{equation}
 It follows that $\theta$ stochastically dominates a natural-valued random variable $\widetilde{\theta}$ such that ${\mathop{\rm P}}(\widetilde{\theta} > N) = c N^{-(1-1/\alpha+\widetilde{\varepsilon})}$ for $N \geq 1$. Consequently, we obtain the bound:$${\mathop{\rm P}}(S_{\theta}([C n^{1-\frac{1}{\alpha}+\varepsilon}]) < n) \leq {\mathop{\rm P}}(S_{\widetilde{\theta}}([C n^{1-\frac{1}{\alpha}+\varepsilon}]) < n), \quad n \geq 1.$$
 The random variable $\widetilde{\theta}$ belongs to the domain of attraction of a positive stable distribution with parameter $(1-1/\alpha+\widetilde{\varepsilon})^{-1}$. Thus, for any $\varepsilon > \widetilde{\varepsilon}$:
 $$\frac{S_{\widetilde{\theta}}([C n^{1-\frac{1}{\alpha}+\varepsilon}])}{n} \overset{{\mathop{\rm P}}}{\to} +\infty, \quad n \to \infty,$$and consequently,$$\limsup_{n\to\infty} {\mathop{\rm P}}(S_{\theta}([C n^{1-\frac{1}{\alpha}+\varepsilon}]) < n) \leq \lim_{n\to\infty} {\mathop{\rm P}}(S_{\widetilde{\theta}}([C n^{1-\frac{1}{\alpha}+\varepsilon}]) < n) = 0.$$ 
 Since $\widetilde \varepsilon>0$ was arbitrary, this completes the proof of Theorem \ref{thm:lim_skew_Levy_walk} (a).    
\end{proof}

\begin{proof}[Idea of the proof of Theorem \ref{thm:lim_skew_Levy_walk} (b)] 

It follows from Lemma \ref{lem:subordinating} that it suffices to prove convergence of the scalings of continuous time Markov chain $\bar X(t)\coloneqq   X(N(t)), t\geq 0,  $   to the skew stable Lévy process, where $N$ is a Poisson process with intensity 1 that is independent of $X.$  

It is well known that weak convergence of Feller processes is equivalent to a strong convergence of their semigroups or resolvents, see Trotter-Sova-Kurtz-Mackevi\v{c}ius theorem \cite{Kallenberg}. In our case  $\frac{  X(N(nt))}{a_n}$ take values in $a_n^{-1}\, {\mathbb Z}$ and $U_{\alpha,\beta}$ take values in ${\mathbb R}$. Nevertheless there are generalizations of the corresponding result, e.g., Theorem \cite[Theorem 2.11, page 172]{Ethier+Kurtz:1986}. We take the formulation of the corresponding generalization 
from \cite{Dong+Iksanov+Pilipenko:2024+}.
\begin{lemma}\label{thm:weak_convergenceEK_chains}
 Let $X^{(0)}$ be a Feller process on ${\mathbb R}$ and, for each $n\in{\mathbb N}$, $X^{(n)}$ a time-homogeneous Markov process on $G_n$, a subset of ${\mathbb R}$, with paths in ${\mathcal D}$, where $G_1$, $G_2,\ldots$ are possibly different. 
For each $n\in{\mathbb N}_0$ and $\lambda>0$, denote by   $R^{(n)}_\lambda$  the resolvent of $X^{(n)}$. 

Assume  that the random variables $X^{(n)} %_n
(0)$ converge in distribution to $X^{(0)} %_0
(0)$ as $n\to\infty$, and  for each $f\in C_0({\mathbb R})$ and each $\lambda>0$,
\begin{equation}
    \label{eq:conv_resolvents_restr}
\lim_{n\to\infty}\sup_{x\in G_n}\,|R^{(n)}_\lambda f(x)-R^{(0)}_\lambda f(x)|=0.
 \end{equation}

\noindent Then
\[
X^{(n)} %_n
~\Rightarrow~ X^{(0)},\quad n\to\infty
\]
in ${\mathcal D}$.
\end{lemma}
We will apply   Lemma \ref{thm:weak_convergenceEK_chains} for $X^{(n)}(t)\coloneqq \frac{  X(N(nt))}{a_n}, \quad G_n\coloneqq (a_n)^{-1}{\mathbb Z},\quad X^{(0)}(t)\coloneqq U_{\alpha,\beta}(t).$

Next, we describe the main heuristic steps of the proof. These steps are designed to provide an intuitive understanding of the underlying probabilistic mechanisms; for a full treatment of the rigorous technical details, the reader is referred to \cite{Dong+Iksanov+Pilipenko:2024+}  or \cite[\S 3.4]{IPMS_book}.

\textit{Step 1.} For any strong Markov process $Y$ we have 
\begin{equation}
    \label{eq:resolvent_general}
R^Y_\lambda f(x)=  V^Y_\lambda f(x)+{\mathop{\rm E}}_x {\rm e}^{-\lambda \sigma_{0}} R^Y_\lambda f(0),\quad x\in{\mathbb R}, 
\end{equation}
where $\sigma_0=\sigma_0(Y)\coloneqq \inf\{t\geq 0 : Y(t)=0\}.$ 
The proof of \eqref{eq:resolvent_general} is a standard argument from Markov processes theory.  

Hence, the proof of the theorem reduces to verifying
\begin{equation}
    \label{eq:conv_res_need0}
    \lim_{n\to\infty}\sup_{x\in G_n} |V_\lambda^{\bar X^{(n)}}f(x)-V_\lambda^{U_\alpha}f(x)|=0, f\in C_0({\mathbb R}),
\end{equation}
\begin{equation}
    \label{eq:conv_res_need00}
    \lim_{n\to\infty}\sup_{x\in G_n} |{\mathop{\rm E}}_x {\rm e}^{-\lambda \sigma_{0}(\bar X^{(n)})}- {\mathop{\rm E}}_x{\rm e}^{-\lambda \sigma_{0}(U_\alpha)}|=0
\end{equation}
and
\begin{equation}
\label{eq:conv_res_need000}
    \lim_{n\to\infty} \big|R_\lambda^{\bar X^{(n)}}f(0) -\frac{\int_{{\mathbb R}}V^{U_\alpha}_\lambda f \; \mathrm{d} m }{\lambda \int_{{\mathbb R}}V^{U_\alpha}_\lambda 1 \; \mathrm{d} m }\big|=0, f\in C_0({\mathbb R}),
\end{equation}

\textit{Step 2.}  It follows from formula \eqref{eq:resolvent_general}  that
\[
V^Y_\lambda f(x) = R^Y_\lambda f(x)- {\mathop{\rm E}}_x {\rm e}^{-\lambda \sigma_{0}} R^Y_\lambda f(0),\quad x\in{\mathbb R}. 
\]
Hence, to calculate $V^Y_\lambda f$ it suffices to know $R^Y_\lambda f$ and ${\mathop{\rm E}}_x {\rm e}^{-\lambda \sigma_{0}}.$
It also may be easily seen from the definition of  $V^Y_\lambda$ that
\[
V^Y_\lambda 1(x) = \lambda^{-1}(1- {\mathop{\rm E}}_x {\rm e}^{-\lambda \sigma_{0}}),\quad x\in{\mathbb R}. 
\]

Denote $\bar S_\xi^{(n)}(t)\coloneqq \frac{S_\xi(N(nt))}{a_n},$ where $S_\xi$ is independent from the Poisson process $N.$ 
Notice that for any $x\in G_n$:
\[
V_\lambda^{\bar S_\xi^{(n)}}f(x)=V_\lambda^{X^{(n)}}f(x), \quad {\mathop{\rm E}}_x {\rm e}^{-\lambda \sigma_{0}(\bar S_\xi^{(n)})}={\mathop{\rm E}}_x {\rm e}^{-\lambda \sigma_{0}(X^{(n)})}.
\]
So, to prove \eqref{eq:conv_res_need0}, \eqref{eq:conv_res_need00} it is sufficient to prove that
\begin{equation}
\begin{aligned}
        \label{eq:conv_res_need1}
    &\lim_{n\to\infty}\sup_{x\in G_n} |R_\lambda^{\bar S_\xi^{(n)}}f(x)-R_\lambda^{U_\alpha}f(x)|=\\
    &\lim_{n\to\infty}\sup_{x\in G_n} |{\mathop{\rm E}}\int_0^\infty {\rm e}^{-\lambda t} \Big(f(x+\bar S_\xi^{(n)}(t))-f(x+{U_\alpha}(t)\Big) \mathrm{d} t|=0,\quad f\in C_0({\mathbb R}),
\end{aligned}
\end{equation}
and
\begin{equation}
    \label{eq:conv_res_need2}
    \lim_{n\to\infty}\sup_{x\in G_n} |{\mathop{\rm E}}_x {\rm e}^{-\lambda \sigma_{0}(\bar S_\xi^{(n)})}- {\mathop{\rm E}}_x {\rm e}^{-\lambda \sigma_{0} (U_\alpha)}|=0.
\end{equation}

Convergence \eqref{eq:conv_res_need1} follows from uniform continuity of $f$, convergence \eqref{eq:conv}, and Skorokhod's representation theorem.

The convergence in \eqref{eq:conv_res_need2} is a significantly more delicate matter than that in \eqref{eq:conv_res_need1}. This difficulty arises from the fact that the weak convergence of stochastic processes does not, in general, imply the convergence of their respective hitting times.

To circumvent this, one can utilize the explicit formulas for the Laplace transforms of hitting times for both Lévy processes and random walks:
  \begin{equation}\label{eq:generating_hitting1}
{\mathop{\rm E}}_x {\rm e}^{-\lambda \sigma_{0} (\bar S_\xi^{(n)})}= \frac{ u^{(n)}_\lambda(-x)}{  u^{(n)}_\lambda (0)}, 
\end{equation} 
where
 \[
 u^{(n)}_\lambda(x) =\int_0^\infty {\rm e}^{-\lambda t}{\mathop{\rm P}}(\bar S_\xi^{(n)}(t)=-x){\rm d}t.
 \]
and 
\begin{equation}\label{eq:Laplica_hitting_Levy}
{\mathop{\rm E}}_x {\rm e}^{-\lambda \sigma(U_\alpha)}= \frac{u_\lambda(-x)}{u_\lambda(0)},\quad x\in\mathbb{R}, \ \lambda>0,
\end{equation}
 where  $u_\lambda(x,y)=u_\lambda(y-x)$, $x,y\in\mathbb{R}$  is the density of the resolvent kernel of $U_\alpha$.
 
Formula  \eqref{eq:generating_hitting1} can be proved by elementary methods, while its continuous counterpart \eqref{eq:Laplica_hitting_Levy} is a deep result from potential theory of Lévy processes, see Corollary 18 on p.~64 in \cite{Bertoin:1996}.
 
 It is known  that
\begin{equation}\label{eq:resolvent_stable}
u_\lambda(x)=\frac{1}{2 \pi} \int_{-\infty}^\infty\frac{\cos(x\theta)}{\lambda+\theta^\alpha} {\rm d}\theta,\quad x\in\mathbb{R},
\end{equation}
and
\begin{equation}\label{eq:return_walk}
u^{(n)}_\lambda(x)=\frac{1}{ 2\pi (\lambda+n) }
\int_{-\pi a_n}^{\pi a_n} \frac{{\rm e}^{{\rm i}\theta x}}{\lambda+n(1-{\mathop{\rm E}}{\rm e}^{i\xi \theta/a_n})}{\rm d}\theta,
\end{equation}
e.g.,  \cite[Theorem 19(iii) ]{Bertoin:1996} for the general Lévy process and \cite[Corollary 3.4.1.]{IPMS_book} for  random walks.

 It follows from our assumption  on $\xi$ that $\lim_{n\to\infty}(n(1-{\mathop{\rm E}}{\rm e}^{i\xi \theta/a_n}))=|\theta|^\alpha$.  Therefore, convergence \eqref{eq:conv_res_need2} appears plausible.   However, transition to  the uniform convergence requires considerable  analytical effort.

\textit{Step 3.} To calculate $R_\lambda^{\bar X^{(n)}}f(0)$, note that $\bar X^{(n)}$ behaves like $\bar S_\xi^{(n)}$ until it reaches 0, then remains at 0 for an exponentially distributed random time with parameter $n$, then makes a jump with distribution $\eta/a_n,$ then moves like $\bar S_\xi^{(n)}$ until it reaches 0, and so on. 
The next  formula follows from (2.2)  on p.~137 in \cite{Blumenthal2012excursions} for `holding and jumping' Markov processes:
\begin{equation}
    \label{eq:4ResplventHoldJump}
R_\lambda^{\bar X^{(n)}}f(0)= 
\lambda^{-1}\, \frac{n^{-1}f(0)+{\mathop{\rm E}} [(V^{\bar S_\xi^{(n)}}_\lambda f) (\eta/a_n)]}{n^{-1}+{\mathop{\rm E}}[(V^{\bar S_\xi^{(n)}}_\lambda 1) (\eta/a_n)]}=\lambda^{-1}\, \frac{n^{-1}f(0)+\int_{{\mathbb R}} (V^{\bar S_\xi^{(n)}}_\lambda f) (y) \mathrm{d} F_{\eta/a_n}(y)}{n^{-1}+\int_{{\mathbb R}} (V^{\bar S_\xi^{(n)}}_\lambda 1) (y) \mathrm{d} F_{\eta/a_n}(y)}.
\end{equation}

Since  $\Big(\frac{S_\eta(n)}{b_n}\Big)$ converges in distribution to a $\beta$-stable law with the Lévy measure $m(\mathrm{d} y)=
(c_-{\mathds{1}}_{(-\infty,\, 0)}(y)+c_+{\mathds{1}}_{(0,\,\infty)}(y))|y|^{-(1+\beta)}{\rm d}y,\quad y\in\mathbb{R}$, we have convergence
\begin{equation}
    \label{eq:1691}
    \lim_{n\to\infty}\frac{{\mathop{\rm P}}(\eta/a_n\in[x_1, x_2])}{{\mathop{\rm P}}(|\eta|>a_n)}=m([x_1,x_2])
\end{equation}
for any  interval $[x_1,x_2]$ that does not contain 0. Hence, it is natural to expect that \eqref{eq:1691} and \eqref{eq:conv_res_need00} imply that 
\begin{equation}
    \label{eq:1698}
    \lim_{n\to\infty}\frac{\int_{{\mathbb R}}  V^{\bar S_\xi^{(n)}}_\lambda f  (y) \mathrm{d} F_{\eta/a_n}(y)}{{\mathop{\rm P}}(|\eta|>a_n)}=\int_{\mathbb R}  V^{U_\alpha}_\lambda f  (y) m(\mathrm{d} y).
\end{equation} 
In fact, this is true for any bounded and uniformly continuous function $f$. This fact requires careful analytical justification, since the measure $m$ is unbounded in any neighborhood of 0. Furthermore, to prove that the integral on the right-hand side is well-defined, we need to know some asymptotic bounds for $V^{U_\alpha}_\lambda f$ in the neighborhood of 0.

Recall that $a_n=o(n), n\to \infty$ since $(a_n)$ is a regularly varying sequence with parameter $\alpha^{-1}\in(1/2,1).$  Hence, 
\eqref{eq:1698} and \eqref{eq:4ResplventHoldJump} imply \eqref{eq:conv_res_need000}. This completes the proof of Theorem \ref{thm:lim_skew_Levy_walk} (b).

 \end{proof}

%%%%%%%%%%%%%%%%%%%%%%%%%%%%%%%%%%%%%%%%%%%%%%%%%%%%%%%%%%

\textbf{Discussion and possible generalizations.}

\begin{enumerate}
    \item The only property of $\eta$ required for the proof of Theorem \ref{thm:lim_skew_Levy_walk} (a) is the limit established in \eqref{eq:|eta|small}. This property is inherited if $|\eta|$ is stochastically dominated by a random variable already satisfying this condition. For instance, the result holds if there exists a positive random variable $\widetilde{\eta}$ such that the tail ${\mathop{\rm P}}(\widetilde{\eta} > x)$ is regularly varying at $+\infty$ with index $\beta$, and the jump sizes satisfy the stochastic bound:
    $${\mathop{\rm P}}(|\eta| > x) \leq {\mathop{\rm P}}(\widetilde{\eta} > x), \quad x \geq 1.$$

    \item Consider a sequence of random walks $(X_n(k))_{k\geq 0}$ with transition probabilities
    $$ p^n_{x,y}=\begin{cases}
        {\mathop{\rm P}}(\xi_n=y-x), & x\neq 0;\\
        {\mathop{\rm P}}(\eta_n =y), & x=0, 
    \end{cases}$$
    where $\xi_n$ and $\eta_n$ are integer-valued random variables. Assume that there are sequences $(a_n)$ and $(b_n)$ such that $\big(\frac{S_{\xi_n}(nt)}{a_n}\big)_{t\geq 0}$ and $\big(\frac{S_{\eta_n}(nt)}{b_n}\big)_{t\geq 0}$ converge in distribution to Lévy processes $(U(t) )_{t\geq 0}$ and $(V(t))_{t\geq 0}$, where $V$ is a pure jump Lévy process with an infinite Lévy measure. 
    
    To what extent can Theorem \ref{thm:lim_skew_Levy_walk} be generalized to this case? 
    
    To establish an analogue of Theorem \ref{thm:lim_skew_Levy_walk} (a), it is sufficient to demonstrate that the number of returns to the origin for the process $X_n(k)$ (where $0 \leq k \leq nT$) is $o(b_n)$ as $n \to \infty$. The explicit formulas \eqref{eq:generating_hitting1} and \eqref{eq:return_walk} are not restricted to the specific laws previously discussed; they are applicable to any random walk.

    The generalization of Theorem \ref{thm:lim_skew_Levy_walk} (b) is significantly more complex and far from trivial. A primary requirement is that the origin must not be a polar set for the limit process $U$ (see, e.g., \cite[\S II.5]{Bertoin:1996} or \cite[\S 43]{Sato}). Under this condition, the integrals in \eqref{eq:RES}---with $U_\alpha$ replaced by $U$---must converge. If convergence holds, the analogue of \eqref{eq:RES} defines the resolvent of a Feller process, see \cite[Theorem 2.8, Chapter V]{Blumenthal2012excursions}. The fundamental formulas used in our proof, such as \eqref{eq:generating_hitting1}, \eqref{eq:return_walk}, and \eqref{eq:4ResplventHoldJump}, are inherently universal. While formal computations suggest the convergence of \eqref{eq:conv_res_need2} or \eqref{eq:1698}, establishing a rigorous proof remains a formidable task. Even in the specific case treated in Theorem \ref{thm:lim_skew_Levy_walk}, the proof was remarkably involved. The remaining challenges are entirely analytical, centering on the validity of passing to the limit under the integral sign and establishing the uniform convergence of the relevant integrals. While it is evident that a broad class of limit theorems could be derived using this framework, such developments lie beyond the scope of the present article.

    \item Assume that the random walk $S_\xi$ is perturbed within a finite set $A = \{-m, \dots, m\}$. In this setting, Theorem \ref{thm:lim_skew_Levy_walk} (a) can be generalized provided it can be shown that for any $\widetilde{\varepsilon} > 0$, there exist constants $c > 0$ and $x \in \mathbb{Z} \setminus A$ such that:
    $${\mathop{\rm P}}_x(\sigma_A(S_\xi) > N) \geq \frac{c}{N^{1-1/\alpha+\widetilde{\varepsilon}}}, \quad N \geq 1,$$
    cf. \eqref{eq: long return}, where $\sigma_A(S_\xi) = \inf\{k \in \mathbb{N} : S_\xi(k) \in A\}$ denotes the first hitting time of the set $A$. According to the Kesten--Spitzer ratio theorem (see \cite[Theorem 4a]{KestenSpitzer1963}), the following limit exists:
    $$\lim_{n\to\infty} \frac{{\mathop{\rm P}}_x(\sigma_A(S_\xi) > n)}{{\mathop{\rm P}}_0(\sigma_0(S_\xi) > n)} = g_A(x) \in [0, \infty).$$
    Unfortunately, we know of no simple proof that $g_A(x)>0$; for further discussion and related reasoning on this topic, we refer the reader to \S 3.5.1 of \cite{IPMS_book}.

    We conjecture that Theorem \ref{thm:lim_skew_Levy_walk} (b) can also be generalized to the case $A=\{-m,\dots,m\}$ if $\beta\in(0,\alpha-1)$ and 
    \[
    {\mathop{\rm P}}_i(\pm X(1)>n)\sim c_\pm^i n^{-\beta} l(n), \quad n\to+\infty
    \]
    for each $i \in A$, where $l$ is a function slowly varying at infinity. Apparently, the resolvent of the limit process has the form \eqref{eq:RES}, where
    \[
    m(\mathrm{d} y) = \sum_{i\in A}\pi_i m_i(\mathrm{d} y)= \sum_{i\in A}\pi_i (c_-^i{\mathds{1}}_{y<0 }+c_+^i{\mathds{1}}_{y>0})|y|^{-(1+\beta)}{\rm d}y,
    \]
    and $(\pi_i)_{i\in A}$ is a stationary distribution of a certain ``entrance Markov chain'' in $A$. In this setup, many technical difficulties arise in all the formulas used in the proof, since it is necessary to take into account which point $X$ jumps to when entering $A$.

    \item Are there other methods of proof for results like Theorem \ref{thm:lim_skew_Levy_walk} (b)? 

    The answer is yes, though one should not expect such proofs to be significantly shorter or less technical. For instance, a perturbed random walk can be viewed as a function of the unperturbed walk and its perturbations (see constructions in Section \ref{sec:spider}). One might then utilize the representation of the skew stable Lévy process through Itô's excursion theory and attempt to establish convergence by constructing the processes on a single probability space. Incidentally, the stochastic differential equation  for the skew stable Lévy process was previously derived using methods from Itô's theory \cite{IksanovPilipenko2021skewLevy, IPMS_book}. However, Itô's excursion theory is quite complex, and we see no clear analytical shortcuts available in this direction.

    Currently, the martingale problem characterization for the skew stable Lévy process remains an open question. Similarly, the answer remains unknown even in the case $\alpha=2$, where $U$ is a Brownian motion but the perturbation $\eta$ belongs to the domain of attraction of a stable law with $\beta \in (0, 1)$. A particularly interesting direction for future research would be to derive the transition densities of the skew stable Lévy process as the fundamental solution to a certain partial differential equation with boundary conditions at the origin and to extend the corresponding theory to the multidimensional case, where general Feller-Wentzell boundary conditions are imposed on a manifold.
\end{enumerate}

 %rata
{\bf Acknowledgments.}  
   The author  thanks  the Swiss National Science Foundation for partial support  of the paper (grants IZRIZ0\_226875, 200020\_214819, 200020\_200400, and 200020\_192129). Support from the EPSRC (grant EP/Z000580/1) and the Isaac Newton Institute for Mathematical Sciences, Cambridge, is gratefully acknowledged, particularly for the hospitality during the programme Stochastic systems for anomalous diffusion.   I express my sincere gratitude to Ilya Pavlyukevich for helpful discussions and valuable comments.% Finally, I am grateful to the anonymous reviewer for their comments and suggestions, which helped to improve the  paper.

\end{document}